\documentclass[10pt]{amsart}
\allowdisplaybreaks

\usepackage{amsfonts}
\usepackage{amsmath}
\usepackage{amssymb}
\usepackage{amsthm}
\usepackage{amsbsy}
\usepackage{xcolor}
\usepackage{graphicx} \usepackage{enumerate} \usepackage{multicol}
\usepackage{mathrsfs} \usepackage[all,cmtip]{xy}
\usepackage{todonotes}
\newcounter{dummy} \numberwithin{dummy}{section}
\newtheorem{theorem}[dummy]{Theorem}
\newtheorem{corollary}[dummy]{Corollary}
\newtheorem{lemma}[dummy]{Lemma}
\newtheorem{definition}[dummy]{Definition}
\newtheorem{proposition}[dummy]{Proposition}
\theoremstyle{remark}
\newtheorem{remark}[dummy]{Remark}

\newcommand{\calL}{\mathcal{L}}

\newcommand\E{\mathbb{E}}

\DeclareMathOperator{\End}{End}
\DeclareMathOperator{\Dom}{Dom}

\DeclareMathOperator{\id}{id}

\DeclareMathOperator{\pr}{pr}
\DeclareMathOperator{\rank}{rank}
\DeclareMathOperator{\spn}{span}
\DeclareMathOperator{\tr}{tr}

\DeclareMathOperator{\Ric}{Ric}

\newcommand{\ptr}{/ \! /}
\newcommand{\hptr}{/ \! \hat{/}}

\DeclareMathOperator{\SO}{SO}

\DeclareMathOperator{\so}{\mathfrak{so}}

\DeclareMathOperator{\bnabla}{\boldsymbol{\nabla}}
\DeclareMathOperator{\Dev}{Dev}
\DeclareMathOperator{\gap}{gap}
\DeclareMathOperator{\e}{e}
\newcommand{\FC}{\mathscr{F}C^\infty}
\newcommand{\FCc}{\mathscr{F}C^\infty_0}
\newcommand{\ve}{\varepsilon}

\def\fatdot{{\text{\LARGE.}}}

\numberwithin{equation}{section}

\title[Functional inequalities on path space of sub-Riemannian manifolds]{Functional inequalities on path space of sub-Riemannian manifolds and applications}
\author[L.J.~Cheng, E.~Grong and A.~Thalmaier]{Li-Juan Cheng\textsuperscript{1,2}, Erlend Grong\textsuperscript{3} and Anton Thalmaier\textsuperscript{1}}

  \address{\textsuperscript{1}University of Luxembourg, Department of Mathematics, 
Maison du Nombre, L--4364 Esch-sur-Alzette, Luxembourg}
  \address{\textsuperscript{2}Department of Applied Mathematics, Zhejiang University of Technology,
    Hangzhou 310023, The People's Republic of China}
  \address{\textsuperscript{3}University of Bergen, Department of Mathematics, P.O.~Box 7803, 5020 Bergen, Norway}

  \email{lijuan.cheng@uni.lu \text{\rm and} chenglj@zjut.edu.cn}
  \email{erlend.grong@gmail.com}
  \email{anton.thalmaier@uni.lu}

\subjclass[2010]{60J60, 58J65, 53C17}

\keywords{Bounded Ricci curvature, path space, functional inequalities, sub-Riemannian geometry}

\begin{document}

\begin{abstract}
For sub-Riemannian manifolds with a chosen complement, we first establish the derivative formula and integration by parts formula on path space  with respect to a natural gradient operator. By using these formulae, we then show that upper and lower bounds of the horizontal Ricci curvature correspond to functional inequalities on path space analogous to what has been established in Riemannian geometry by Aaron Naber, such as gradient inequalities, log-Sobolev and Poincar\'{e} inequalities.
\end{abstract}

\maketitle
%Version: \today

%\tableofcontents

\section{Introduction}
Stochastic analysis on the path space over a complete Riemannian manifold
 has been well developed ever since B. K. Driver
\cite{Dri92} proved the quasi-invariance theorem for the Brownian motion
on compact Riemannian manifolds in 1992. A key point of the study is to first
establish an integration by parts formula for the associated gradient operator
induced by the quasi-invariant
flows, then prove functional inequalities
for the corresponding Dirichlet form (see e.g. \cite{Fang,Hsu97} and references within). For more analysis on Riemannian
path spaces we refer to \cite{Elworthy-Li,Malliavin,Stroock} and references within. Recently, there has been an extensive study by A.~Naber \cite{Naber} on the equivalence of bounded Ricci curvature and certain inequalities on path space. R.~Haslhofer and A.~Naber \cite{HN} extended these
results to characterize solutions of the Ricci flow, see also
\cite{HN1}.

In the present article, we develop this formalism in the framework of hypoelliptic operators and diffusions in sub-Riemannian geometry. Let $(M,H,g)$ be a sub-Riemannian manifold, meaning that $H$ is a subbundle of $TM$ with a metric tensor~$g$. Let $\nabla$ be an affine connection on $TM$ compatible with $(H,g)$ in the sense that it preserves $H$ and its metric~$g$ under parallel transport. We define an operator
\begin{equation} \label{LH} L =  \tr_H \nabla^2_{\times, \times},\end{equation}
as the trace of the Hessian $\nabla^2$ over $H$ with respect to the inner product $g$.
We assume that the subbundle $H$ is bracket-generating, meaning that its sections and their iterated brackets span the entire tangent bundle. This makes $L$ into a hypo\-elliptic operator on functions by H\"ormander's theorem \cite{Hor67}. Let $B_t^x$ be a standard Brownian motion in the inner product space $H_x$. Then the solution of the SDE,
\begin{equation*}  dX_t^x = \ptr_t \circ dB_t^x, \quad X_0^x = x\end{equation*}
is a diffusion on $M$ with $\frac{1}{2} L$ as infinitesimal generator, with $\ptr_t\colon T_x M \to T_{X_t^x}M$ denoting $\nabla$-parallel transport along $X_t^x$. We note that for the case when $H = TM$ and $\nabla$ is the Levi-Civita connection, the operator $L$ is the Laplacian and $X_t^x$ is the Brownian motion in $M$.

The analysis of path space of sub-Riemannian manifolds has been earlier considered in \cite{BF15,BGF17} for the case where the sub-Riemannian structure $(H,g)$ is the restriction to the transverse bundle of a foliation that is Riemannian, totally geodesic and of Yang-Mills type. In this present paper, we will generalize the approach in \cite{BGF17} to arbitrary sub-Riemannian manifolds with a metric preserving complement, which include sub-Riemannian manifolds coming from Riemannian foliations, but does not require anything of the metric along the foliation or even any extension of the sub-Riemannian metric.

The derivative formula and integration by parts formula established in this paper correspond to a generalization of the gradient in \cite{BGF17} to a more general class of sub-Riemannian manifolds. This gradient is defined in terms of a connection $\nabla$ compatible with the sub-Riemannian structure which is canonical in the sense that any choice of complement $V$ to the subbundle $H$ determines it uniquely. To motivate the reasonability of the definition, we first review the smooth path space and the development map with respect to an arbitrary connection. The underlying idea is that if we have a variation of curves $\{ \gamma^s\}$ that are all tangent to~$H$, then the corresponding variational vector field $Y=\partial_s \gamma^s|_{s=0}$ will not be in $H$ in general, yet it can not be arbitrary in the sense that is determined by $\pr_H Y$ for any choice of projection $\pr_H\colon TM \to H$. We will construct the sub-Riemannian gradient on path space to reflect this property.

Our formula for the damped gradient appears more similar to the definition in Riemannian geometry, however, it uses the adjoint connection $\hat \nabla$ of $\nabla$ which will not be compatible with the sub-Riemannian structure. We will show that there are nice formulas relating the two operators with the gradient operator defined by the adjoint connection,  which help us to establish the derivative and integration by parts formulae for both the gradient and the damped gradient.

Having set up this formalism, we extend the approach of Naber to the sub-Riemannian case in our main result in Theorem~\ref{main-th2}. We establish  functional inequalities on the path space of the stochastic flow $x \mapsto X_t^x$ including gradient inequalities, log-Sobolev inequalities and Poincar\'{e} inequalities. These inequalities are shown to be equivalent to bounds on \emph{the horizontal Ricci operator} $\Ric_H\colon H \to H$ which is defined taking the trace  of the curvature tensor only over $H$. We want to emphasize that it is quite surprising that we can establish almost identical relations between bounded Ricci curvature and functional inequalities in the sub-Riemannian path space as in the Riemannian case. By contrast, the relationship between lower Ricci curvature bounds and functional inequalities for the heat semigroup is much more complicated in the sub-Riemannian case compared to the Riemannian one, see e.g. \cite{BaGa17,BBG,BKW16,GT16a,GT16b} for details.

The structure of the paper is as follows. In Section~\ref{sec:Smooth} we first consider the smooth path space and development with respect to an arbitrary connection. We review the basic definitions of sub-Riemannian manifolds and connections compatible with such structures. Unlike in the Riemannian case, we do not have torsion-free compatible connections on such spaces, however, we give analogues of the Levi-Civita connection by defining a canonical connection with minimal torsion relative to a chosen complement $V$ to the horizontal bundle $H$. We finally use these connections to define corresponding vector fields on smooth path space.

We generalize the definition of these vector fields in Section~\ref{sec:Gradients} in order to define a gradient and a damped gradient for functions on path space. We relate these concepts and look at their properties in Theorems~\ref{IBPD},~\ref{IBPtildeD} and \ref{Dev-Intr-formulas}. In particular, we establish integration by parts formulas for both the gradient and the damped gradient, generalizing the Riemannian case and the case treated in \cite{BGF17}. Finally, in Section~\ref{sec:BoundCurv}, we show that several functional inequalities related to functions on path space are equivalent to the analogue of bounded Ricci curvature. We state our main result in Theorem~\ref{main-th2}. From this result, we also obtain a spectral gap estimate in Corollary~\ref{cor} for the Ornstein-Uhlenbeck operator corresponding to the gradient.

In Section~\ref{sec:Geometry}, we look closer at how such results can be interpreted geometrically. Intuitively, we show that if one uses the canonical connection $\nabla$ corresponding to a metric preserving complement $V$, then the sub-Riemannian path space has geometry ``similar to $M/V$''. This latter concept is well defined in the case when $V$ is an integrable submanifold corresponding to a regular foliation $\Phi$ in which $M/\Phi$ has an induced Riemannian structure, but our formalism is valid for non-integrable choices of complements as well.

For the main results of this paper, we need to choose a complement which is metric preserving. To explain the reasons behind this assumptions for this choice and for later references, we include some formulas related to a general choice of connection and complement in Appendix~\ref{sec:Any}, which show the additional complications that appear if the complement is not metric preserving.

\

\paragraph*{\bf Acknowledgements}
The first author was supported by the National Natural Science
Foundation of China (Grant No. 11901604), the second author 
in part by the Research Council of Norway (project number 249980/F20),
and the third author by the Fonds National
de la Recherche Luxembourg (Open project O14/7628746 GEOMREV).
We thank Fabrice Baudoin for interesting discussions regarding the topic.

\section{Smooth path space and sub-Riemannian geometry} \label{sec:Smooth}
\subsection{Smooth path space and development} \label{sec:SmoothPath}
An affine manifold is a pair $(M, \nabla)$ where $\nabla$ is an affine connection on $TM$. Let $\mathbf{T}$ denote the torsion of $\nabla$, i.e.
$$\mathbf{T}(Y,Z) = \nabla_Y Z - \nabla_Z Y - [Y,Z], \quad Y,Z \in \Gamma(TM),$$
and let $\mathbf{R}$ denote its curvature
$$\mathbf{R}(Y_1, Y_2) Z = \left( \nabla_{Y_1} \nabla_{Y_2} - \nabla_{Y_2} \nabla_{Y_1} - \nabla_{[Y_1,Y_2]} \right) Z, \quad Y_1, Y_2, Z \in \Gamma(TM).$$
We define its \emph{adjoint} $\hat \nabla$ as the connection
\begin{equation} \label{Adjoint}
\hat \nabla_Y Z = \nabla_Y Z - \mathbf{T}(Y,Z).
\end{equation}
Observe that the torsion of $\hat \nabla$ is $- \mathbf{T}$ and hence $\nabla$ is the adjoint of $\hat \nabla$. We remark that if $(s,t) \mapsto \omega_t^s$ is a two-parameter function with values in $M$, then
\begin{equation} \label{NablaHatNabla}
D_s \frac{\partial}{\partial t} \omega^s_t = \hat D_t \frac{\partial}{\partial s} \omega^s_t,
\end{equation}
where $D_s$ and $\hat D_t$ denote covariant derivatives of respectively $\nabla$ in the direction of~$s$ and $\hat \nabla$ in the direction of $t$.

Let $W^\infty_x(M)$ denote the space of smooth curves 
$[0, \infty) \to M$, $t \mapsto \omega_t$
satisfying $\omega_0 = x$. When $\ptr_t\colon T_x M \to T_{\omega_t} M$ denotes parallel transport with respect to $\nabla$ along a given path $\omega \in W^\infty_x(M)$, we say that $u \in W_0^\infty(T_xM)$ is \emph{the anti-development} of $\omega_t$ if it is the unique solution of
\begin{equation} \label{omegau} \dot u_t = \ptr_t^{-1} \dot \omega_t, \quad u_0 = 0,\end{equation}
with $\dot u_t =\frac{d}{dt} u_t$.
Conversely, we say that $\omega$ is \emph{the development} of $u$. We write $\Dev(u) = \omega$ and $\Dev^{-1}(\omega) = u$. We note that $\Dev^{-1}$ is defined for any element in $W^\infty_x(M)$, however, for a general $u \in W_{0}^\infty(T_x M)$,  $t\mapsto \Dev(u)_t$ might be only defined for short time. If $\omega_t =\Dev(u)_t$ is defined for all time for any $u \in W^\infty_0(T_xM)$, $x \in M$, then $\nabla$ is called \emph{complete}. For the rest of this subsection, we assume that $\nabla$ is complete. For the general case, see Remark~\ref{re:non-complete}. The next lemma describes the derivative of $\Dev$.
\begin{lemma} \label{lemma:SmoothPathSpace_0}
Let $\omega \in W^{\infty}_x(M)$ be an arbitrary smooth curve with $\Dev^{-1}(\omega) = u$. Consider $\omega^s_t = \Dev(u+ sk)_t$ for $k \in W^\infty_0(T_xM)$ and define
\begin{equation} \label{DerY} Y_t = \frac{\partial}{\partial s} \omega^s_t |_{s=0}.\end{equation}
Write $\ptr_t, \hptr_t\colon T_x M \to T_{\omega_t} M$ for parallel transport along $\omega$ relative to respectively $\nabla$ and $\hat \nabla$. If we write
$$Y_t = \ptr_t y_t = \hptr_t \hat y_t$$
with $\hat y_t = \hptr_t^{-1} \ptr_t y_t$, then $y_t$ and $\hat y_t$ are the unique solutions of
\begin{align*}
k_t & =  y_t + \int_0^t \mathbf{T}_{\ptr_s}(y_s, du_s) - \int_0^t\int_0^s \mathbf{R}_{\ptr_r} ( du_r , y_r) du_s \\
& = \int_0^t \ptr_s^{-1} \hptr_s \dot {\hat y}_s \, ds  - \int_0^t\int_0^s \mathbf{R}_{\ptr_r} ( du_r , y_r) du_s.
\end{align*}
\end{lemma}
We remark that in the above statement, we used the notation
$$\mathbf{T}_{\ptr_t}(w_1,w_2) = \ptr_t^{-1} \mathbf{T}(\ptr_t w_1,\ptr_t w_2), \quad \mathbf{R}_{\ptr_t}(w_1, w_2) w_3= \ptr_t^{-1} \mathbf{R}(\ptr_t w_1,\ptr_t w_2) \ptr_t w_3 .$$
We will use this notation for tensors in general throughout the paper.

\begin{proof}
Let $D$ and $\hat D$ be the covariant derivative of respectively $\nabla$ and $\hat \nabla$. Write $e_{1,s}(t), \dots, e_{n,s}(t)$ for an orthonormal $\nabla$-parallel basis along $t \mapsto \omega^{s}_t$ with $e_{j,0}(t) = e_j(t)$ and $e_{j,s}(0) = e_j(0)$, and use the same basis to define $u_t + s k_t = \sum_{j=1}^n (u_j(t) + sk_j(t)) e_{j}(0)$. Then
$$\hat D_t \partial_s \omega =D_s \partial_t \omega = \sum_{j=1}^n \dot k_j e_{j,s}+ \sum_{j=1}^n (\dot u_j+s \dot k_j) D_s e_{j,s}.$$
By definition, we have $D_s e_{j,s}(0) =0$. Furthermore, we have
$$D_t D_s e_{j,s} = (D_t D_s - D_s D_t ) e_{j,s} = \mathbf{R}\left( \frac{\partial}{\partial t} \omega, \frac{\partial}{\partial s} \omega \right) e_{j,s} .$$
It follows that at $s=0$,
$$\hat D_t Y  = \hptr_t \dot{\hat y}_t= D_t Y - \mathbf{T}(\dot \omega_t, Y_t) = \sum_{j=1}^n \dot k_j e_{j}(t) + \sum_{j=1}^n \dot u_j ( D_s e_{j,s}(t) |_{s=0}),$$
and hence,
\begin{align*}
& \ptr_t^{-1} \hptr_t \dot {\hat y}_t = \dot y_t- \mathbf{T}_{\ptr_t}(\dot u_t, y_t) = \dot k_t + \sum_{j=1}^n \dot u_j(t) \int_0^t \mathbf{R}_{\ptr_s}(\dot u_s, y_s ) e_j(0) ds.
\end{align*}
\end{proof}

\begin{remark}[Non-complete connections] \label{re:non-complete}
Let $\omega \in W^\infty_x(M)$ be any given curve with $u = \Dev^{-1}(\omega)$. Then for arbitrary $k \in W_0^\infty(T_xM)$ and any $T > 0$, there is some $\ve >0$ such that $t \mapsto \Dev(u + s k)_t$ has a solution on $[0,T]$ for $|s|< \ve$. Hence, we have that $t \mapsto Y_t$ can still be defined as a derivative of a two-parameter family as in \eqref{DerY} for any $t \geq 0$.
\end{remark}

\subsection{Sub-Riemannian manifolds} We consider a sub-Riemannnian manifold as a triple $(M, H,g)$ where $M$ is a connected manifold, $H \subseteq TM$ is a subbundle of the tangent bundle and $g = \langle \, \cdot \, , \, \cdot \, \rangle_g$ is a metric tensor on~$H$. The sub-Riemannian structure $(H,g)$ induces a map $\sharp: T^*M \to H\subseteq TM$ defined by
$$\langle \alpha, v \rangle = \langle \sharp \alpha, v \rangle_{g}, \quad \alpha \in T^*_xM, \, v \in H_x, \, x \in M.$$
We can then define a (degenerate) sub-Riemannian cometric $g^*$ by
$$g^*(\alpha, \beta) = \langle \alpha, \beta \rangle_{g^*} = \langle \sharp \alpha, \sharp \beta \rangle_{g}.$$
We remark that in what follows, we use $g$, the map $\sharp$ as well as the cometric $g^*$ to state our results. For $v \in H$ and $\alpha \in T^*M$, we also use the notation $|v|_g = \langle v, v \rangle_g^{1/2}$ and $|\alpha|_{g^*} = \langle v, v \rangle_{g^*}^{1/2}$ and ask the reader to keep in mind that $|\alpha|_{g^*}$ may vanish for non-zero covectors.

As usual, we assume that~$H$ is bracket-generating, meaning that sections of~$H$ and their iterated brackets span the entire tangent bundle. A curve $\omega_t$ is called horizontal if it is absolutely continuous and satisfies $\dot \omega_t \in H_{\omega_t}$ for almost every~$t$. The bracket-generating condition implies that any pair of points can be connected by a horizontal curve. We hence have a well defined distance on $M$ given by
\begin{equation} \label{cc-distance} d_g(x,y) =\inf \left\{ \int_0^T \langle \dot \omega_t, \dot \omega_t \rangle_g^{1/2}\, dt \colon\ \omega_0 = x,\ \omega_T = y,\ \omega_t \text{ is horizontal}\right\}.\end{equation}
The topology induced by the metric $d_g$ coincides with the manifold topology. We say that $(M,H,g)$ is complete if $(M, d_g)$ is a complete metric space. For more details on sub-Riemannian geometry, see e.g. \cite{Mon02}.

\subsection{Connections compatible with the sub-Riemannian structure and development}
Let $\nabla$ be an affine connection on $TM$ for a sub-Riemannian manifold $(M,H, g)$. We are interested in the following types on connections.
\begin{definition}
\begin{enumerate}[\rm (a)]
\item We say that $H$ is parallel with respect to $\nabla$ if $H$ is preserved under parallel transport. Equivalently, $H$ is parallel with respect to $\nabla$ if for any $Y \in \Gamma(H)$, $Z \in \Gamma(TM)$ and $x \in M$, we have $\nabla_Z Y|_x \in H_x$. It is called horizontally parallel if $H$ is preserved by parallel transport along horizontal curves or equivalently if $\nabla_Z Y|_x \in H_x$ for any $x \in M$, $Y,Z \in \Gamma(H)$.
\item We say that $\nabla$ is compatible with  the sub-Riemannian structure $(H,g)$ if orthonormal frames in $H$ are taken to orthonormal frames in $H$ under parallel transport along all smooth curves. Equivalently, $\nabla$ is compatible with $(H,g)$ if $H$ is parallel with respect to $\nabla$ and for any $Y_1, Y_2 \in \Gamma(H)$, $Z \in \Gamma(TM)$, we have
\begin{equation} \label{compatible} Z \langle Y_1, Y_2 \rangle_g = \langle \nabla_Z Y_1, Y_2 \rangle_g + \langle Y_1, \nabla_Z Y_2 \rangle_g.\end{equation}
We say that $\nabla$ is horizontal compatible with $(H,g)$ if $H$ is horizontally parallel with respect to $\nabla$ and \eqref{compatible} holds for $Z \in \Gamma(H)$.
\end{enumerate}
\end{definition}
Unlike what holds in Riemannian geometry, there exists no affine connection that is both compatible with the sub-Riemannian structure and also torsion free when $H$ is bracket-generating and a proper subbundle of $TM$, see e.g. \cite{GT16}.

Let $t \mapsto \omega_t$ be any smooth horizontal curve with $\omega_0 = x$. If $H$ is horizontally parallel relative to $\nabla$, then the corresponding anti-development $u = \Dev^{-1}(\omega)$ is a smooth curve $H_x$, and the converse is also true for any curve $u \in W^\infty_x(H_x)$ if only for short time in general. We say that $\nabla$ is \emph{horizontally complete} if $\Dev(u)$ is defined for all time for every $u \in W^\infty_0(H_x)$. We note the following relation.

\begin{proposition}
Let $(M, H,g)$ be a complete sub-Riemannian manifold and let $\nabla$ be a connection that is horizontally compatible with the sub-Riemannian structure $(H,g)$. Then $\nabla$ is horizontally complete.
\end{proposition}
\begin{proof}
Let $u \in W^\infty_0(H_x)$ be fixed. For a given $T > 0$, let $\varphi(T) = \int_0^T |\dot u|^2_{g_x} dt$ denote the length of $u$ up to time $T$. Let $[0,T]$ be some interval for which the solution of
$$\dot u =\ptr_t^{-1} \dot \omega_t, \quad \omega_0 = x$$
exists. Then since $\ptr_t^{-1}$ is a linear isometry by our assumptions, we have that $\omega_t$, $t \in [0,T]$ has to be contained in the ball $B_g(x, \varphi(T)+\ve)$, $\ve >0$, centered at $x$ with radius $\varphi(T)+\ve$ defined relative to the sub-Riemannian distance $d_g$ defined in \eqref{cc-distance}. Since we are assuming that $(M,H,g)$ is complete, all such balls have compact closures, see e.g.~\cite{Bel96}. Hence, for any $T >0$, we can solve the development equation in $B_g(x, \varphi(T)+\ve)$. It follows that $\Dev(u)$ is well defined.
\end{proof}

We finally note that the map $\Dev$ restricted to $W^\infty_0(H_x)$ only depends on parallel transport along horizontal curves. For this reason, we will consider the concept of partial connections. Investigating this concept also allows us to find a unique choice of horizontally compatible connection relative to a choice of complement.

\subsection{Partial connections on sub-Riemannian manifolds}
\emph{A partial connection $\bnabla$ on $H$ in the direction of $H$} is a map $\bnabla\colon \Gamma(H) \times \Gamma(H) \to \Gamma(H)$, $(Y, Z) \mapsto \bnabla_Y Z$ satisfying that for $f\in C^{\infty}(M)$,
$$\bnabla_{fY} Z = f\bnabla_Y Z \quad \text{ and } \quad \bnabla_Y f Z = (Yf) Z + f \bnabla_Y Z.$$
In other words, covariant derivatives are only defined in the direction of $H$. A partial connection will give us a well defined parallel transport along $H$-horizontal curves. For more on partial connections, see \cite{CGJK15}.

Let $(M, H, g)$ be a sub-Riemannian manifold. A partial connection on $H$ in the direction of $H$ is \emph{compatible with $(H,g)$} if
\begin{equation} \label{PartiallyCompatible}
Z \langle Y_1, Y_2 \rangle_g = \langle \bnabla_Z Y_1, Y_2 \rangle_g + \langle Y_1, \bnabla_Z Y_2 \rangle\end{equation} for any $Z, Y_1, Y_2 \in \Gamma(H)$. We define its torsion $\mathbf{t}\colon H \times H \to TM$ by
$$\mathbf{t}(X,Y) = \bnabla_X Y- \bnabla_Y X - [X,Y].$$

\begin{lemma} Let $\bnabla$ be a partial connection on $H$ in the direction of $H$.
\begin{enumerate}[\rm (a)]
\item The map $X, Y \mapsto \mathbf{t}(X, Y) \mod H$ does not depend on the choice of $\bnabla$. In particular, $\mathbf{t}$ cannot vanish when $H$ is bracket-generating. Furthermore, if $V$ is a choice of complement for $H$, that is $TM = H \oplus V$, with corresponding projection $\pr_V$, then $\pr_V \mathbf{t}$ is independent of choice of partial connection.
\item Assume that $\bnabla$ is compatible with the sub-Riemannian metric. Then it is uniquely determined by its torsion.
\item Let $V$ be a choice of complement to $H$. Then there is a unique partial connection $\bnabla$ compatible with the sub-Riemannian structure $(H,g)$ and with $\mathbf{t}(H,H) \subseteq V$.
\end{enumerate}
\end{lemma}
\begin{proof}
The result in (a) follows from the fact that $\mathbf{t}(X, Y)  =  -[X,Y] \mod H$. To prove (b), choose an arbitrary complement $V$ and a reference compatible partial connection $\bnabla'$. Write its torsion $\mathbf{t}' = \mathbf{t}_H' + \mathbf{t}_V' = \pr_H \mathbf{t}' + \pr_V \mathbf{t}'$. For any partial connection $\bnabla$, we write $\bnabla_X Y = \bnabla'_X Y + \kappa(X) Y$ with torsion $\mathbf{t} = \mathbf{t}_H + \mathbf{t}_V'$. We then have that
$$\kappa(X) Y - \kappa(Y) X = \mathbf{t}(X,Y) - \mathbf{t}'(X,Y) = (\mathbf{t}_H - \mathbf{t}'_H)(X,Y) ,$$
from definition of torsion and using that \eqref{PartiallyCompatible} implies,
$$\langle \kappa(X) Y_1, Y_2 \rangle_g + \langle Y_1, \kappa(X) Y_2 \rangle_g = 0.  $$
Hence, it follows that $\kappa$ is determined by
\begin{align*}
\langle \kappa(X) Y_1, Y_2 \rangle & = \frac{1}{2}  \langle  (\mathbf{t}_H - \mathbf{t}'_H)(X,Y_1) , Y_2 \rangle_g \\
& \qquad  - \frac{1}{2} \langle  (\mathbf{t}_H - \mathbf{t}'_H)(Y_1,Y_2) , X \rangle_g - \frac{1}{2} \langle  (\mathbf{t}_H - \mathbf{t}'_H)(X,Y_2) , Y_1 \rangle_g. \end{align*}
Hence $\kappa$ is uniquely determined by $\mathbf{t}_H$. Furthermore, to prove (c), if we take
$$\langle \kappa(X) Y_1, Y_2 \rangle = \frac{1}{2} \left(- \langle  \mathbf{t}'_H(X,Y_1) , Y_2 \rangle_g + \langle  \mathbf{t}'_H(Y_1,Y_2) , X \rangle_g + \langle  \mathbf{t}'_H(X,Y_2) , Y_1 \rangle_g\right),$$
then $\mathbf{t}_H = 0$ and this is the unique such choice.
\end{proof}

Given a sub-Riemannian manifold $(M,H,g)$, let $V$ be a choice of complement. Let $\pr_H$ and $\pr_V$ be the corresponding projections. We write $\bnabla^{g,V}$ for the unique compatible partial connection with $\mathbf{t}(H,H) \subseteq V$. We will also write $\nabla = \nabla^{g,V}$ for an affine connection on the following form
$$\nabla_X Y = \left\{ \begin{array}{ll} \bnabla_X^{g,V} Y, & \text{if $X,Y \in \Gamma(H)$}; \\
\pr_H [X, Y], & \text{if $X \in \Gamma(V),\ Y \in \Gamma(H)$}; \\
\pr_V [X,Y], & \text{if $X \in \Gamma(H),\ Y \in \Gamma(V)$}, \end{array} \right.$$
and where $\nabla_{|V} |V$ can be an arbitrary partial connection on $V$ in the direction of~$V$. We allow this ambiguity in covariant derivatives, since, in what follows, the values of covariant derivatives of sections of $V$ in the directions of $V$ will not affect our formulas. We note the following.
\begin{proposition} \label{pro:NablaGV}
The connection on the form $\nabla = \nabla^{g,V}$ with torsion $\mathbf{T}$ satisfies the following properties:
\begin{enumerate}[\rm (a)]
\item Both $H$ and $V$ are parallel with respect to $\nabla$;
\item $\nabla$ is horizontally compatible with $(H,g)$;
\item $\mathbf{T}(H, H) \subseteq V$;
\item $\mathbf{T}(H, V) = 0$.
\end{enumerate}
Conversely, any connection $\nabla$ satisfying {\rm (a)--(d)} is of the form $\nabla^{g,V}$.
\end{proposition}
We also note that $\nabla^{g,V}$ is always horizontally compatible with $(H,g)$, and hence is horizontally complete if $(H,g)$ is complete. In order to get compatible connections, we consider the following property introduced in \cite{GT16a}.
\begin{definition} \label{def:MetricPreserving}
For any $Z \in \Gamma(TM)$, let $\mathcal{L}_Z$ denote the corresponding Lie derivative. Let $V$ be a complement to a sub-Riemannian manifold $(M,H,g)$ with corresponding projection $\pr_H$ and $\pr_V$. The complement $V$ is then called metric preserving if for any $Z \in \Gamma(V)$ and $X \in \Gamma(H)$, we have
$$(\calL_Z \pr^*_H g)(X,X)  =0.$$
\end{definition}
One verifies from the definition of $\nabla = \nabla^{g,V}$ that $\nabla$ is compatible with $(H,g)$ if and only if $V$ is metric preserving.

\begin{remark}
Let $(M,H,g)$ be a sub-Riemannian manifold. We say that a Riemannian metric $\bar{g}$ \emph{tames} $g$ if $\bar{g}| _{H \times H}= g$. Assume that we have chosen a taming metric $\bar{g}$ for which $V = H^\perp$ is the orthogonal complement of $H$. Assume further that $V$ is integrable with corresponding foliation $\Phi$. Then the assumption of $V$ being metric-preserving is equivalent to assuming that the metric $\bar{g}$ is bundle-like, or, in a different terminology, assuming that $\Phi$ is a Riemannian foliation. We refer to \cite{GT16a} for details. We emphasize that none of these properties depend on $\bar{g}|_{V \times V}$.
\end{remark}

\subsection{The smooth horizontal path space seen from a metric preserving complement} \label{sec:HorPathSmooth}
Let $(M, H, g)$ be a complete sub-Riemannian manifold with a chosen complement $V$. We assume that $V$ is metric preserving such that $\nabla = \nabla^{g,V}$ is a connection compatible with the sub-Riemannian structure. For discussion of the general case, see Appendix~\ref{sec:Any}.

Let $\nabla$ have torsion $\mathbf{T}$ and curvature $\mathbf{R}$ and define $\Dev = \Dev^\nabla$ relative to this connection. Then the following result holds.

\begin{proposition} \label{lemma:SmoothPathSpace}
Consider $\omega^s_t = \Dev(u+ sk)_t$ for $u,k \in W^\infty_0(H_x)$. Write $\omega = \Dev(u)$ and introduce a linear map $A_t = A_t^\omega: T_xM \to T_xM$ by
\begin{equation}
A_t w = \int_0^t \mathbf{T}_{\ptr_s} (du_s, w).
\end{equation}
If $Y_t = \frac{\partial}{\partial s} \omega_t^s |_{s=0}$, then $Y_t = \ptr_t y_t = \hptr_t \hat y_t$ with
$$y_t = h_t + \int_0^t dA_s h_s, \quad \hat y_t = h_t - \int_0^t A_s dh_s,$$
where $h_t = \pr_H y_t$ is the unique solution of
\begin{align} \label{kandh}
k_t & = h_t  - \int_0^t\int_0^s \mathbf{R}_{\ptr_r} \left( du_r , h_r \right) du_s. \end{align}
\end{proposition}
The statement is a special case of Lemma~\ref{lemma:SmoothPathSpaceAny}, Appendix. Based on this result, we make the following definition.
\begin{definition} \label{Def:SmoothDh}
Let $(M,H,g)$ be a sub-Riemannian manifold with a metric preserving complement $V$ and define $\nabla = \nabla^{g,V}$. For any $h \in W^\infty_0(H_x)$, we define a vector field $D_h$ on $W^\infty_x(M)$ by
$$D_h|_\omega = \ptr_t^{-1} \left( \textstyle h_t + \int_0^t \mathbf{T}_{\ptr_s}(du_s, h_s)  \right)= \ptr_t^{-1} \left( \textstyle h_t + \int_0^t dA_s h_s  \right), \quad u = \Dev^{-1}(\omega),$$
where $\ptr_t$ denotes parallel transport along $\omega$ with respect to $\nabla$.
\end{definition}
We note the following immediate consequence of Proposition~\ref{lemma:SmoothPathSpace}.
\begin{corollary}
For any horizontal curve $\omega$ with $u = \Dev^{-1}(\omega)$, we have 
$$\left\{ \frac{d}{ds} \Dev(u + s k) |_{s=0} \colon\, k \in W_0^\infty(H) \right\} = \{ D_h |_\omega \colon\, h \in W_0^\infty(H) \}.$$
\end{corollary}
We note that $\frac{d}{ds} \Dev(u + s k) |_{s=0} = D_h|_\omega$ where $k$ and $h$ are related by \eqref{kandh}. In the case when $\nabla$ is a flat connection, i.e. if $\mathbf{R} \equiv 0$, then $k = h$. We will generalize such vector fields to functions $h$ with values in the Cameron-Martin space in the next section.

\section{Diffusions and gradients on path space} \label{sec:Gradients}
Throughout this section, we assume that $M$ is compact for a simpler presentation. We hence have that all tensors are bounded and that all local martingales are indeed martingales. The same results hold in the non-compact case under some additional assumptions, see Section~\ref{sec:Noncompact} for details.
\subsection{Sub-Riemannian diffusions and notation}
Let $M$ be a compact manifold and let $W_x = W_x(M)$ be the space of continuous maps %(
$\omega\colon[0,\infty) \to M$ %]
with $\omega_0 = x$. Let $(H,g)$ be a sub-Riemannian structure on $M$ and let $\nabla$ be a horizontally compatible connection. Recall that the Hessian of $\nabla$ is defined as
$$\nabla_{Y_1, Y_2}^2 = \nabla_{Y_1} \nabla_{Y_2} - \nabla_{\nabla_{Y_1} Y_2}, \quad Y_1, Y_2 \in \Gamma(TM).$$
We write $L =  \tr_H \nabla^2_{\times,\times}$ for the connection sub-Laplacian of $\nabla$ and let  $x \mapsto X_t^x \in W_x$ be the stochastic flow with generator $\frac{1}{2} L$ and $X_0^x = x$ defined on the filtered probability space $(W_x , \mathscr{F}_{\!\fatdot}, \mathbb{P}_x )$.

For $0 \leq s \leq t < \infty$, let $\ptr_{s,t} : T_{X_s^x} M \to T_{X_t^x} M$ denote the parallel transport along $X_t^x$ with respect to $\nabla$ and write $\ptr_{0,t} = \ptr_t$. Note that $\ptr_{s,t} = \ptr_t \ptr_s^{-1}$. The solution $B_t^x$ of
$$dB_t^x = \ptr_t^{-1} \circ d X_t^x, \quad B_0^x = 0 \in H_x,$$
is a standard Brownian motion in $H_x$. Hence, $X_t^x$ can be considered as the development of the Brownian motion in $H_x$.

For any $T > 0$, we define $W^T_x$ as the curves in $W_x$ restricted to $[0, T]$. We write the induced structure of a filtered probability space as $(W_x^T, \mathscr{F}^T_{\!\fatdot}, \mathbb{P}_{x}^T )$ and the corresponding stochastic process as $X^x_{[0,T]}$. Introduce the Cameron-Martin space $\mathbb{H}^T_x := \mathbb{H}^T(H_x)$ as the Hilbert space of absolutely continuous functions $h\colon [0, T] \to H_x$ with $\int_0^T |\dot h_t|^2_{g} \,dt < \infty$ and with inner product
$$\langle h, k \rangle_{\mathbb{H}} = \int_0^T \langle \dot h_t , \dot k_t \rangle_{g}\,dt, \quad h,k \in \mathbb{H}^T_x.$$
More generally, we define
\begin{align*}
\mathbb{H}^T_{W,x} & = L^2(W^T_x \to \mathbb{H}^T_x;  \mathscr{F}_{\!\fatdot}^T , \mathbb{P}_{x}^T )\\
& = \left\{ h \in L^2(W^T \to \mathbb{H}^T_x; \mathbb{P}_x^T)\colon\, \text{$h_t$ is $\mathscr{F}_t$-measurable, $t \in [0,T]$} \right\},\end{align*}
as a Hilbert space with inner product $\langle h, k \rangle_{L^2} = \mathbb{E} \langle h, k \rangle_{\mathbb{H}}$. As usual, we write $\langle h, B^x \rangle_{\mathbb{H}} = \int_0^T \langle \dot h_s, dB_s^x \rangle_{g}$.

\subsection{Gradient on path space}
Let $V$ be an arbitrary complement to $H$ and define $\nabla = \nabla^{g,V}$. Let $x \mapsto X_{\fatdot}^x$ be the corresponding stochastic flow with generator equal to the trace of the Hessian of $\nabla$. For any $h \in W^\infty_0(H_x)$, recall the definition of $D_h$ on $W^\infty_x(M)$ from Definition~\ref{Def:SmoothDh}. As parallel transport is well defined along a path in $W_x$ almost surely, we can consider $D_h$ as a $\mathbb{P}$-almost surely defined vector field on $W_x$. We want to make this definition more precise and valid for functions $h$ in the Cameron-Martin space.

We first introduce the following notation. Let $\mathbf{T}$ be the torsion of $\nabla$ and define $\delta_H\mathbf{T}(\, \cdot \,) = - \tr_H (\nabla_\times \mathbf{T})(\times, \, \cdot \,)$. We define the following endomorphism $A_{x,t} = A_t: T_xM \to T_xM$ by
$$A_t(\, \cdot \,) = \int_0^t \mathbf{T}_{\ptr_s}( \circ dB_s^x, \, \cdot \,) =  \int_0^t \mathbf{T}_{\ptr_s}( dB_s^x, \, \cdot \,) - \frac{1}{2}\int_0^t (\delta_H \mathbf{T})_{\ptr_s}(\, \cdot \,) ds,$$
where $\circ d$ denotes the Stratonovich differential and $\delta_H \mathbf{T} = - \tr_H (\nabla_\times \mathbf{T})(\times, \, \cdot \,)$.
We remark that by the defining properties of $\nabla$ in Proposition~\ref{pro:NablaGV}, we have that $A_t(H_x) \subseteq V_x$ and $A_t(V_x) = 0$. For fixed $T > 0$, consider the space of cylindrical functions
$$\FC = \left\{ F\colon\, \omega \in W^T_x \mapsto f(\omega_{t_1}, \dots, \omega_{t_n})\left| \begin{array}{c} 0 \leq t_1 < \cdots < t_n \leq T, \\ n \geq 0,\ f \in C^\infty(M^n) \end{array} \right.\right\}.$$
We define $D_h$ acting on a cylindrical function $F : \omega \mapsto f(\omega_{t_1}, \omega_{t_2}, \dots,  \omega_{t_n})$ by
\begin{align*}
D_h F & = \sum_{i=1}^n  \langle \ptr_{t_i}^{-1} d_if|_{(\omega_{t_1}, \dots, \omega_{t_n})}, {\textstyle h_{t_i} + \int_0^{t_i} dA_t h_t} \rangle \\
& =\sum_{i=1}^n  \langle \ptr_{t_i}^{-1} d_if|_{(\omega_{t_1}, \dots, \omega_{t_n})}, {\textstyle  \int_0^{t_i} (\id + A_{t_i} - A_t) dh_t} \rangle . \end{align*}
This is consistent with our definition in the smooth case in Section~\ref{sec:HorPathSmooth}. If we define $D_tF \in H_x$ by
$$D_t F := \sum_{i=1}^n  1_{t \leq t_i}  \sharp (\id + A_{t_i} - A_t )^* \ptr_{t_i}^{-1}   d_if|_{(\omega_{t_1}, \dots, \omega_{t_n})},$$
then for every $h \in \mathbb{H}^T_x$,
\begin{align*}
& \int_0^T \langle D_t F, \dot h_t \rangle_g dt = D_h F.
\end{align*}
We define \emph{the gradient} $DF \in \mathbb{H}^T_{W,x}$ by the relation $\langle DF,h\rangle_{\mathbb{H}} = D_hF$.

We like to show that the operator $D\colon F \to DF$ can be closed on path space in the case when $V$ is a metric preserving complement. For this case, we need the following integration by parts formula. Recall that $\mathbf{R}$ is the curvature of~$\nabla$. Introduce the corresponding Ricci operator $\Ric\colon TM \to TM$ by
\begin{equation} \label{Ric} \Ric(v) = - \tr_H \mathbf{R}(\times, v ) \times.\end{equation}\goodbreak

\begin{theorem} \label{IBPD}
Assume that $V$ is metric preserving.
\begin{enumerate}[$(a)$]
  \item For any $F \in \FC$, we have
   $$d_x\E_x[F]=\E_x[D_0F]-\frac{1}{2}\int_0^T\E_x[(\Ric_{\ptr_s}Q_s)^*D_sF]\,ds,$$
   where $Q_t$ is the solution to the following equation:
   \begin{align}\label{Q-equation}
    dQ_t = - \frac{1}{2} \Ric_{\ptr_t} Q_t \,dt, \qquad Q_0 = \id_{T_xM}.
   \end{align}
  \item For any $F \in \FC$ and $h \in \mathbb{H}^T_x$, we have
$$\mathbb{E}_x[\langle DF, h \rangle_{\mathbb{H}}] =\mathbb{E}_x\left[F \int_0^T \langle \dot h_t + \frac{1}{2} \Ric_{\ptr_t} h_t, dB_t \rangle_{g}\right]. $$
\end{enumerate}
\end{theorem}
In particular, for $F(X_{[0,T]})=f(X_t)$, our result reduces to the following form, see the end of Section~\ref{sec: Proof:Dev-Intr-formulas} for more details.
\begin{corollary}\label{cor-Dev-Int}
Assume that $V$ is metric preserving.
\begin{enumerate}[\rm (a)]
  \item Let $Q_t$ be the solution of \eqref{Q-equation}.
   $$dP_t f(v) = \E_x\left[ \left\langle \ptr_t^{-1} df|_{X_t} , Q_t v + \int_0^t dA_s Q_s v\right\rangle \right], \quad v \in T_xM;$$
  \item for any $k \in \mathbb{H}^T_x$ with $h_t = Q_t \int_0^t Q_s^{-1} dk_s$, we have
\begin{align*}
\E_x\left[ f(X_T)\langle k, B \rangle_{\mathbb{H}} \right] & = \E_x\left[ \left\langle \ptr_T^{-1} df|_{X_T}  , h_T + \int_0^T dA_t h_t  \right\rangle \right] .\end{align*}
\end{enumerate}

\end{corollary}
We note that the result in (a) has already appeared in \cite{GT16}. We show this formula by proving the corresponding derivative formula and integration by parts formula in Theorem~\ref{Dev-Intr-formulas} for the damped gradient.

\subsection{The damped gradient on path space} We define \emph{the damped gradient} $\tilde DF$ similarly to the formula in Riemannian geometry, but using parallel transport of the adjoint connection. We use the connection $\nabla = \nabla^{g,V}$ and define $\Ric$ as in \eqref{Ric}. Define $\hptr_{s,t}: T_{X_s} M \to T_{X_t}M$ as parallel transport along $X_t$ with respect to~$\hat \nabla$, the adjoint of~$\nabla$, and write $\hptr_t = \hptr_{0,t}$. We first introduce $\hat Q_{s,t} : T_{X_s} M \to T_{X_s} M$, $s \leq t$,
$$\frac{d}{dt} \hat Q_{s,t} = - \frac{1}{2}\Ric_{\hptr_{s,t}}  \hat Q_{s,t}, \quad \hat Q_{s,s} = \id_{T_{X_s} M}.$$
We note that if $\hat Q_t = \hat Q_{0,t}$, then $\hat Q_{s,t} =  \hptr_s \hat Q_t \hat Q_s^{-1} \hptr_s^{-1}$ and for $s\leq r\leq t$,
$$\hat Q_{s,t} = \hptr_{s,r}^{-1} \hat Q_{r,t}  \hptr_{s,r}   \hat Q_{s,r}.$$
For $F \in \FC$ with $F(\omega) = f(\omega_{t_1}, \dots, \omega_{t_n})$, we define
$$\tilde{D}_tF(\omega) :=  \sum_{i=1}^n 1_{t\leq t_i} \sharp \ptr_{t}^{-1} \hat Q_{t,t_i}^* \hptr_{t, t_i}^{-1} d_if|_{(\omega_{t_1}, \dots, \omega_{t_n})},$$
and furthermore, for any $k \in \mathbb{H}^T_x$,
$$\tilde D_k F: = \langle \tilde DF, k \rangle_{\mathbb{H}} := \int_0^T \tilde D_t F dk_t.$$
The next result clarifies the relationship between $DF$ and $\tilde D F$.

\begin{theorem} \label{IBPtildeD}
Assume that $V$ is metric preserving. Let $Q_t: T_xM \to T_x M$ be the solution to
$$Q_0 = \id_{T_xM}, \quad dQ_t = - \frac{1}{2} \Ric_{\ptr_t} Q_t dt.$$
\begin{enumerate}[\rm (a)]
\item For any $k \in \mathbb{H}^T_x$ and $F \in \FC$, if $h_t = Q_t \int_0^t Q_s^{-1} dk_s$, then
$$\tilde D_k F = D_h F,$$
\item For any $F \in \FC$,
$$\tilde D_t F= D_tF - \frac{1}{2} \int_t^T  (\Ric_{\ptr_s} Q_s Q_t^{-1})^* D_s F ds.$$
\end{enumerate}
\end{theorem}
Next,
with respect to the damped gradient, we  establish the gradient formula and
integration by parts formula on path space as follows.
\begin{theorem}\label{Dev-Intr-formulas}
Assume that $V$ is metric preserving.
\begin{enumerate}[\rm (a)]
\item (Derivative formula) For any $F\in \FC$ and $t >0$, we have
$$D_t \E_x[F|\mathscr{F}_t]=\E_x[\tilde D_tF|\mathscr{F}_t].$$

\item (The Clark-Oc\^one formula) For any $F \in \FC$, we have
$$F = \E_x[F] + \int_0^T \langle  \E_x[\tilde D_s F | \mathscr{F}_s] , dB_s^x \rangle.$$
\item (Integration by parts formula) For any $k \in \mathbb{H}_x^T$ and $F \in \FC$,
$$\mathbb{E}_x[\langle \tilde DF, k \rangle_{\mathbb{H}}] =\mathbb{E}_x\left[F \langle k, B \rangle_{\mathbb{H}} \right]. $$
In particular, for any $k \in \mathbb{H}^T_x$ with $h_t = Q_t \int_0^t Q_s^{-1} dk_s$, we have
\begin{align} \label{IBPHatNabla}
\E\left[ f(X_T^x)\langle k, B \rangle_{\mathbb{H}} \right] & = \E\left[ \left\langle \ptr_T^{-1} df|_{X_T^x}  , U_T \hat Q_T \int_0^T \hat Q_s^{-1} U_s^{-1} dk_s  \right\rangle \right].\end{align}
\end{enumerate}
\end{theorem}

We will prove Theorem~\ref{IBPtildeD} and Theorem~\ref{Dev-Intr-formulas} in the next subsections. Now we show how Theorem~\ref{IBPD} follows from these results.
\begin{proof}[Proof of Theorem \ref{IBPD}]
We can prove (a) directly by using Theorem \ref{Dev-Intr-formulas} (a) and  Theorem \ref{IBPtildeD} (b). For (b),
\begin{align*}
\E_x D_hF&=\E_x\tilde{D}_kF=\mathbb{E}\left[F \langle k, B \rangle_{\mathbb{H}} \right]\\
&=\mathbb{E}_x\left[F \int_0^T \langle \dot h_t + \frac{1}{2} \Ric_{\ptr_t} h_t, dB^x_t \rangle_{g}\right],
\end{align*}
where the last equation follows from
$dk_t= dh_t + \frac{1}{2} \Ric_{\ptr_t} h_t dt$.
\end{proof}

\subsection{Proof of Theorem~\ref{IBPtildeD}} Let $(M,H,g)$ be a sub-Riemannian manifold with a metric preserving complement $V$. In all the steps below, we will consider the connection $\nabla = \nabla^{g,V}$.

 Define the tensor $\Ric$ relative to $\nabla$ as in~\eqref{Ric}. To prove Theorem \ref{IBPtildeD}, we first observe the $\Ric$ vanish outside the horizontal bundle $H$.
\begin{lemma}\label{P-Ric}
Write $\Ric_H = \Ric|_H$. Then
$$\Ric v = \Ric_H \pr_H v = \pr_H \Ric_H \pr_H.$$
\end{lemma}
\begin{proof}
We note that since $H$ is parallel with respect to $\nabla$, we have $\mathbf{R}( \, \cdot \, , \, \cdot \,) v \in H_x$ for any $v \in H_x$, $x \in M$. It follows that $\Ric(TM) \subseteq H$. From the proof of Lemma~\ref{lemma:SmoothPathSpaceAny}, we also have that $\langle \mathbf{R}(v_1, z)v_1, v_2 \rangle_g = 0$ whenever $V$ is metric preserving, giving us $\Ric(V) = 0$.
\end{proof}
Next, we relate the damped gradient and the gradient by the following conversion formula.
\begin{lemma} \label{lemma:Conversion}
For any element in $k \in \mathbb{H}_x^T$, define $h_t = Q_t \int_0^t Q_s^{-1} dk_s$. We then have that
$$\ptr_t^{-1} \hptr_t \hat Q_t \int_0^t \hat Q_s^{-1} \hptr_s^{-1} \ptr_s dk_s = h_t + \int_0^t dA_s h_s.$$
\end{lemma}
\begin{proof}
Define $U_t = \ptr_t^{-1} \hptr_t$. We note first that $dU_t =  \mathbf{T}_{\ptr_t}(\circ dB_t, U_t \, \cdot \,)$ and $U_0 = \id$, giving us that $U_t = \id + A_t$. Since $A_t^2 = 0$, we have that $U_t^{-1} = \id - A_t$. We will use this to find a formula for $\hat Q_t$ by
\begin{align*}
d \hat Q_t &= - \frac{1}{2} \Ric_{\hptr_t} \hat Q_t \, dt= - \frac{1}{2} (\id- A_t) \Ric_{\ptr_t} (\id + A_t) \hat Q_t \, dt \\
& = - \frac{1}{2} \Ric_{\ptr_t} \hat Q_t \, dt + \frac{1}{2} A_t \Ric_{\ptr_t} \hat Q_t \, dt.
\end{align*}
Hence, we have that $\hat Q_t = Q_t + \frac{1}{2} \int_0^t A_s \Ric_{\ptr_s} Q_s \, ds = Q_t - \int_0^t A_s dQ_s$. Since $Q_t w =w$ for any $w \in V$, the inverse of $\hat Q_t$   is $\hat Q_t^{-1}= (\id + \int_0^t A_s dQ_s) Q_t^{-1} $.

We use these identities to compute
\begin{align}\label{hatQ-eq}
U_t \hat Q_t = Q_t + \int_0^t dA_r Q_r, \quad (U_t \hat Q_t)^{-1} = Q_t^{-1} - \int_0^t dA_r Q_r Q_t^{-1},
\end{align}
and hence,
\begin{align} \label{UQ}
(d(U_t \hat Q_t))(U_t \hat Q_t)^{-1}  & = \left(  -\frac{1}{2} \Ric_{\ptr_t} Q_t dt+ dA_t Q_t \right) (U_t Q_t)^{-1} \\
 & =  -\frac{1}{2} \Ric_{\ptr_t} dt+ dA_t . \nonumber
\end{align}
If we write $a_t =  U_t \hat Q_t \int_0^t \hat Q_s^{-1} U_s^{-1} dk_s$, then
$$da_t = d(U_t \hat Q_t) (U_t \hat Q_t)^{-1} a_t + dk_t = -\frac{1}{2} \Ric_{\ptr_t} a_t dt+ dA_t a_t + dk_t, \quad a_0 = 0. $$
It follows that $a_t =  h_t + \int_0^t dA_s h_s$ since $h_t$ is the solution of $dh_t + \frac{1}{2} \Ric_{\ptr_t} h_t dt = dk_t$.
\end{proof}
Using this lemma, we prove Theorem \ref{IBPtildeD} as follows.
\begin{proof}[Proof of Theorem \ref{IBPtildeD}]
We consider $F \in \FC$. Note that
\begin{align*}
\int_0^T\tilde D_t F(\omega) dk_t & =\int_0^T \sum_{i=1}^n 1_{t \leq t_i} \langle \ptr_{t}^{-1} \hat Q_{t,t_i} \hptr_{t,t_i}^{-1} d_i f|_{(\omega_{t_1}, \dots, \omega_{t_n})} , d k_t \rangle_g \\
& = \int_0^T \sum_{i=1}^n 1_{t \leq t_i} \langle \ptr_{t_i}^{-1} d_i f|_{(\omega_{t_1}, \dots, \omega_{t_n})} ,  U_{t_i}\hat Q_{t_i} \hat Q_t^{-1} U_t^{-1} dk_t \rangle_g .
\end{align*}
Hence, we have that $\int_0^T \langle \tilde D_t F(\omega) , \dot{k} \rangle_g \,ds= \int_0^T \langle D_t F, \dot{h} \rangle_g\,ds$ by Lemma \ref{lemma:Conversion}, which then gives (a).

The relationship between $D_tF$ and $\tilde{D}_tF$ can be observed from
\begin{align*}
\int_0^T(\tilde{D}_tF-D_tF)\dot{k}_t \,dt&=-\frac{1}{2}\int_0^TD_tF \Ric_{\ptr_t}Q_t\int_0^tQ_s^{-1}dt\\
&=-\frac{1}{2}\int_0^T\int_s^TD_tF\Ric_{\ptr_t}Q_tQ_s^{-1}\,dtdk_s,
\end{align*}
which then implies (b).
\end{proof}

\subsection{Proof of Theorem~\ref{Dev-Intr-formulas}} \label{sec: Proof:Dev-Intr-formulas}
We first prove Theorem~\ref{Dev-Intr-formulas}~(a) for the case $n=1$.
\begin{lemma} \label{lemma:DampGradient}
For any $x \in M$, consider $\hat Q_t^x = \hat Q_t$ as the solution of
$$\frac{d}{dt} \hat Q_{t} = - \frac{1}{2} \Ric_{\hptr_t} \hat Q_t\,, \qquad \hat Q_0  = \id_{T_xM},$$
and let $\hat Q_t^*: T^*_xM \to T^*_xM$ be its dual. Define $U_t = \ptr_t^{-1} \hptr_t$.
Then we have 
$$dP_t f|_x = \E\left[ \hat Q_t^* \hptr_t^{-1} df|_{X_t^x}\right] =  \E\left[ \hat Q_t^* U_t^{-1} \ptr_t^{-1} df|_{X_t^x}\right].$$

\end{lemma}
\begin{proof}
For $t \in [0, T]$, consider the $T_x^* M$-valued process
$$\tilde N_s = \hat Q_s^* \hptr_s^{-1} dP_{t-s}f|_{X_s}.$$
 From \eqref{GenWeitzHatNabla} and the fact that $V$ is metric preserving, it follows that $\tilde N_s$ is a local martingale
\begin{eqnarray*}
d\tilde N_s & = &  \hat Q_s^* \hptr_s^{-1} \hat \nabla_{\ptr_s dB_s} dP_{t-s} f|_{X_s},
\end{eqnarray*}
and from our compactness assumption, it is a true martingale.
\end{proof}

\begin{proof}[Proof of Theorem \ref{Dev-Intr-formulas}]
 For part (a),
write $F(\omega)=f(\omega_{t_1},\dots, \omega_{t_n})$. We first consider the case when $t=0$. Write $\Phi(x) = \E_x[F]$. Then we need to prove
\begin{align*}
\sharp d\Phi=\E_x[\tilde D_0 F].
\end{align*}
By Lemma~\ref{lemma:DampGradient}, the desired assertion holds for $n=1$. Assume that
it holds for $n \geq 1$. We will prove that the assertion also holds for $n+1$. First let
\begin{align}\label{def-l}
g(x)=\E [f(x, X_{t_2-t_1}^x,X_{t_3-t_1}^x,\ldots, X_{t_{n+1}-t_1}^x)],\quad x\in M.
\end{align}
By our induction hypothesis,
\begin{align}\label{derivative-l}
dg(x)=\sum_{i=1}^{n+1}\E \left[\hat Q_{0,t_i-t_1}\hptr_{t_i-t_1}^{-1}d_if|_{(x, X_{t_2-t_1}^x,\ldots, X_{t_{n+1}-t_1}^x)}\right]
\end{align}
 for all $x\in M$. Using the strong Markov property of $X$, we look at the process $X$
 starting from $X_{t_1}^x$ at time $t_1$. From \eqref{derivative-l} and from the result at $n=1$,
\begin{align}\label{t0-derivative-formula}
&d \E f|_{X_{t_1}^x, X_{t_2}^x,\cdots, X_{t_{n+1}}^x}=d \E[g(X_{t_1}^x)]=\E \left[\hat Q_{0,t_1}\hptr_{t_1}^{-1}d g|_{X_{t_1}^x}\right]\notag\\
&=\sum_{i=1}^{n}\E \left[\hat Q_{0,t_1} \hptr_{t_1}^{-1} \hat Q_{t_1,t_i}\hptr_{t_1,t_i}^{-1}d_if|_{(X_{t_1}^x, X_{t_2}^x,\ldots, X_{t_{n+1}}^x)}\right] \notag\\
&=\sum_{i=1}^{n}\E \left[\hat Q_{0,t_i} \hptr_{t_i}^{-1}d_if|_{(X_{t_1}^x, X_{t_2}^x,\ldots, X_{t_{n+1}}^x)}\right] .
\end{align}
This completes the proof for $t=0$.

For $F(\omega) = f(\omega_{t_1}, \dots, \omega_{t_n})$, consider $G = \mathbb{E}_x[F| \mathscr{F}_t]$. If $t_{m-1} < t \leq t_{m} $, then $G(\omega) = g(\omega_{t_1}, \dots \omega_{t_{m-1}}, \omega_{t}) = \mathbb{E}_x[ f(\omega_{t_1} ,\dots, \omega_{t_{m-1}}, X_{t_{m}} , \dots, X_{t_n})| X_t=\omega_t]$.
Using the formula \eqref{t0-derivative-formula}, we obtain
\begin{align*}
&d_{\omega_t}g |_{(\omega_{t_1} , \dots, \omega_{t_{m-1}}, \omega_t)}\\
&=d_{\omega_t}\mathbb{E}[ f(\omega_{t_1} ,\dots, \omega_{t_{m-1}}, X_{t_{m}} , \dots, X_{t_n})| X_t=\omega_t]\\
&=d_{\omega_t}\mathbb{E}[ f(\omega_{t_1} ,\dots, \omega_{t_{m-1}-t}, X_{t_{m}-t} , \dots, X_{t_n-t})| X_0=\omega_t]\\
&=\mathbb{E}_{\omega_t}\Big[\sum_{i=m}^n  \hat Q_{0,t_i-t}^* \hptr_{t_i-t}^{-1} d_if |_{(\omega_{t_1} ,\dots, \omega_{t_{m-1}}, X_{t_{m}} , \dots, X_{t_n})}\Big].
\end{align*}
By the strong Markov properties,
\begin{align*}
&\mathbb{E}_{\omega_t}\Big[\sum_{i=m}^n  \hat Q_{0,t_i-t}^* \hptr_{t_i-t}^{-1} d_if|_{(\omega_{t_1} ,\dots, \omega_{t_{m-1}-t}, X_{t_{m}-t} , \dots, X_{t_n-t})}\Big]\\
&=\mathbb{E} \Big[\sum_{i=m}^n \hat Q_{t,t_i}^* \hptr_{t,t_i}^{-1} d_if|_{(\omega_{t_1} ,\dots, \omega_{t_{m-1}}, X_{t_{m}} , \dots, X_{t_n})}\,\big|\,X_t=\omega_t\Big].
\end{align*}
From this equality, we obtain
\begin{align*}
&  D_t \mathbb{E} [F\,|\,\mathscr{F}_t] =  \ptr_t^{-1} d_{\omega_t}g |_{(\omega_{t_1} , \dots, \omega_{t_{m-1}}, \omega_t)} \\
& =   \ptr_t^{-1} \mathbb{E}\Big[\sum_{i=m}^n \hat Q_{t,t_i}^* \hptr_{t,t_i}^{-1} d_if |_{(\omega_{t_1} ,\dots, \omega_{t_{m-1}}, X_{t_{m}} , \dots, X_{t_n})} \,\big|\, X_t=\omega_t\Big]\\
& = \mathbb{E}[ \tilde D_t F | \mathscr{F}_t ].
\end{align*}

For part (b), we first observe that $\E[F\,|\,\mathscr{F}_t]$ is a martingale according to the definition.
By martingale representation, we have
\begin{align*}
d \E[F\,|\, \mathscr{F}_t]&=\langle D_t\E[F\,|\, \mathscr{F}_t], dB_t\rangle = \langle \E[\tilde{D}_tF\,|\, \mathscr{F}_t], dB_t\rangle.
\end{align*}
Integrating from $0$ to $T$ gives
\begin{align}\label{F-formula}
F-\E[F]=\E[F\,|\,\mathscr{F}_T]-\E[F]=\int_0^T\langle \E[\tilde{D}_tF\,|\, \mathscr{F}_t], dB_t\rangle.
\end{align}

For part (c), we first take the formula about $F$ inside the term $\mathbb{E}\left[F \langle k, B^x \rangle_{\mathbb{H}} \right]$. By \eqref{F-formula},
\begin{align*}
\mathbb{E}\left[F \langle k, B^x \rangle_{\mathbb{H}} \right]&=
\mathbb{E}\left[\left(\E[F]+\int_0^T\langle \E[\tilde{D}_tF\,|\, \mathscr{F}_t], dB_t\rangle\right) \langle k, B^x \rangle_{\mathbb{H}} \right]\\
&=\mathbb{E}\left[\int_0^T\langle\E[\tilde{D}_tF\,|\, \mathscr{F}_t], \dot{k}_t \rangle dt  \right],
\end{align*}
giving the formula in (c). By letting $F(X_{[0,T]})=f(X_t)$ we have the equality in \eqref{IBPHatNabla}.
\end{proof}

\begin{proof}[Proof of Corollary \ref{cor-Dev-Int}]
For (a), we first observe from Lemma \ref{lemma:DampGradient}  that
\begin{align*}
\langle dP_tf, v\rangle&=\E\langle\hat{Q}_t^{*}\hptr_t^{-1} df|_{X_t},v \rangle=\E\langle \ptr_t^{-1}df|_{X_t}, U_t\hat{Q}_tv \rangle.
\end{align*}
Then by \eqref{hatQ-eq}, we have
\begin{align*}
\langle dP_tf, v\rangle= \E\left[ \left\langle \ptr_t^{-1} df , Q_t v + \int_0^t dA_s Q_s v\right\rangle \right], \quad v \in T_xM.
\end{align*}
The formula in (b) follows from Theorem \ref{IBPD}\,(b) by taking $F(X_{[0,T]})=f(X_t)$ directly.
\end{proof}

\subsection{Quasi-invariance}
We want to link $\langle \tilde DF, k\rangle_{\mathbb{H}}$ to the directional derivative induced by some quasi-invariant flow. We use techniques from \cite[Chapter~4.2]{Wbook2}. For $k \in \mathbb{H}^{T}_x$
and $s \in (-\ve, \ve)$, let $X^{s} = X^{x,s}$ solve the SDE
\begin{align} \label{X-SDE}
d X_{t}^{s}=\ptr_{t}^{s}\circ \ d B_t
- s \ptr_{t}^{s} d k_t, \qquad X_0^{s}=x,
\end{align}
where  $\ptr^{s}$ is the parallel transport along $X^{s}$ with respect to $\nabla$. This flow is quasi-invariant, i.e.,
the distribution of $X_{[0,T]}^{s}$ is absolutely continuous with respect to that of $X_{[0,T]}= X^x_{[0,T]}$. Let
\begin{align*}
R^{s}= \exp{\left(s \langle k,  B \rangle_{\mathbb{H}}-\frac{s^2}{2} \langle k , k \rangle_{\mathbb{H}}\right)} =\exp{\left(s \int_0^T \langle \dot k, B\rangle_{g}-\frac{s^2}{2}\int_0^T | \dot k_t|^2_{g}\ d t\right)}.
\end{align*}
If $\mathbb{P}$ denotes the Wiener measure on the path space of $H_x$, $d = \rank H_x$, then by the Girsanov theorem,
$$B_t^{s}:=B_t - s k_t$$
is a $d$-dimensional Brownian motion in $H_x$ under the probability measure $\mathbb{P}^s = R^{s}\mathbb{P} = (\xi_{s,k})_* \mathbb{P}$ with
$$\xi_{s,k}(u) = u + sk, \qquad u \in W_0(H_x).$$

\begin{proposition}\label{prop:quasi-invariant}
For any $x\in M$ and $F\in \FC$,
\begin{align*}
 \E_x[\langle \tilde D F, k \rangle_{\mathbb{H}}]=
\lim_{s \to 0}\mathbb{E}\frac{F(X_{[0,T]}^{s})-F(X_{[0,T]})}{s}
\end{align*}
holds for all $k\in \mathbb{H}^T$.
\end{proposition}

\begin{proof}
Let $B_t^{s}=B_t-s k_t$, which is the $d$-dimensional Brownian motion under $\mathbb{P}^s$.
By the weak uniqueness of \eqref{X-SDE}, we conclude that the distribution of
$X$ under $\mathbb{P}^s$  is consistent with that of $X^{s}$
under $\mathbb{P}$. In particular,
$$\mathbb{E} [F(X_{[0,T]}^{s})]=\mathbb{E} \left[R^{s} F(X_{[0,T]})\right].$$
Thus, we have
\begin{align*}
& \lim_{s \to 0}\mathbb{E}\frac{F(X_{[0,T]}^{s})-F(X_{[0,T]})}{s}=\lim_{s\rightarrow0}\mathbb{E}\left[F(X_{[0,T]})\frac{R^{s}-1}{s}\right]\\
&=\mathbb{E}\left[F(X_{[0,T]})\int_0^T\langle \dot k_t, d B_t \rangle_{g}\right]= \E_x[\langle \tilde D F, k \rangle_{\mathbb{H}}].
\end{align*}
\end{proof}

\subsection{Comments for the non-compact case} \label{sec:Noncompact}
In order to keep the exposition simpler, we assumed throughout this section that we are working over a compact manifold. This has the advantage that we can be assured that processes such as $X_t^x$, $A_t$ and $Q_t$ have infinite lifetime, all tensors are bounded and all local martingales that appear in our proofs are indeed true martingales. If one can find alternative ways to show that the same properties hold, our results hold without the compactness assumption. One way to ensure this on a non-compact manifold, is to verify that the following assumptions hold.
\begin{enumerate}[\rm (A)]
\item Assume that $X_t^x$ has infinite lifetime, i.e. assume that $P_t1 =1$.
\item We must be able to pick a taming Riemannian metric $\bar{g}$ such that $\mathbf{T}$ and $\Ric$ are bounded and such that $\sup_{0 \leq t \leq T}|\ptr_t|_{\bar{g} \otimes \bar{g}^*}$ is finite for every $T$. We note that since $\hptr_t = \ptr_t(\id+A_t) $, such assumptions are sufficient to bound parallel transport with respect to $\hat \nabla$ as well.
\item Our cylindrical functions will now be defined as $\FCc$ consisting of functions $F(\omega) = f(\omega_{t_1}, \dots, \omega_{t_n})$ where $f \in C^\infty_0(M^n)$ is of compact support. In order to differentiate expectations of such functions and in order that they remain bounded, we need to assume that $\sup_{t \in [0,T]} | d P_{t}f |_{\bar{g}^*}$ is bounded for every finite $T > 0$ and every compactly supported function $f \in C^\infty_0(M)$.
\end{enumerate}
For sufficient conditions for these assumptions to be satisfied on non-compact manifolds, more specifically on complete sub-Riemannian manifolds coming from totally geodesic Riemannian foliations, see \cite{GT16}.

For alternate approaches of dealing with path space analysis on non-compact spaces in the Riemannian case, we also refer to \cite{ChWu14}. For the rest of the paper, we will assume that properties (A) to (C) are satisfied or some other conditions are satisfied in order to ensure that the results of this section hold.

\section{Bounded curvature and functional inequalities on path space} \label{sec:BoundCurv}

\subsection{Inequalities equivalent to bounded curvature}
Inspired by Naber's work \cite{Naber}, we have the following characterization formulae for the boundedness of $\Ric_H$.
Let $V$ be a metric preserving complement with $\nabla = \nabla^{g,V}$ the corresponding connection. Recall that $\Ric_H:=\Ric|_H$ where $\Ric$ is defined as in \eqref{Ric}. We state the main result of the paper with proof given in the next section.

\begin{theorem}[Characterization of bounded horizontal Ricci curvature by functional inequalities]\label{main-th2} 
Let $K$ be some fixed non-negative constant. Consider the following bound for the horizontal Ricci curvature $\Ric_H$
\begin{align}\label{cur-cond}
 -K \leq \Ric_H \leq K.
  \end{align}
{\rm I.} The following functional inequalities for functions on path space are equivalent to curvature bound \eqref{cur-cond}:
\begin{enumerate}[\rm (i)]
  \item for any $F\in \mathscr{F}C_0^{\infty}$,
   \begin{align*}
   |D_0\mathbb{E}_x[F]|_{g} \leq \mathbb{E}_x\Big[|D_0F|_{g}+ \frac{K}{2} \int_0^T\e^{\frac{K}{2}s}|D_sF|_{g}\, d s\Big];
   \end{align*}
  \item for any $F\in \mathscr{F}C_0^{\infty}$,
   \begin{align*}
   |D_0\mathbb{E}_x[F]|^2_{g} \leq \e^{\frac{K}{2}T}\mathbb{E}_x\Big[|D_0F|^2_{g}+ \frac{K}{2}\int_0^T\e^{\frac{K}{2}s}|D_sF|^2_{g}\, d s\Big];
   \end{align*}

     \item \emph{(Log-Sobolev inequality)} for any $F\in \mathscr{F}C_{0}^{\infty}$ and
    $t>0$ in $[0,T]$,
    \begin{align*}
      \quad&\E_x\Big[\E_x[F^2|\mathscr{F}_{t}]\log \E_x[F^2|\mathscr{F}_{t}]\Big]
             -\E_x[F^2]\log \E_x[F^2]\\
           &\leq 2\int_{0}^{t}\e^{\frac{K}{2}(T-r)}\left(\E_x|{D}_rF|^2_{g}+\frac{K}{2}\int_r^T\e^{\frac{K}{2}(s-r)}\E_x|{D}_sF|^2_{g}\ d s\right)d r;
    \end{align*}
  \item \emph{(Poincar\'e inequality)} for any $F\in \mathscr{F}C_{0}^{\infty}$ and
    $t>0$ in $[0,T]$,
    \begin{align*}
      \quad&\E_x\Big[\E_x[F|\mathscr{F}_{t}]^2\Big]-\E_x[F]^2\\
           &\leq\int_{0}^{t}\e^{\frac{K}{2}(T-r)}\left(\E_x|{D}_rF|^2_{g}+\frac{K}{2}\int_r^T\e^{\frac{K}{2}(s-r)}\E_x|{D}_sF|^2_{g}\ d s\right) d r.
    \end{align*}
 \end{enumerate}
{\rm II.} The following functional inequalities on the manifold $M$ are equivalent to the curvature bound \eqref{cur-cond}:
\begin{enumerate}[\rm (i)] 
\setcounter{enumi}{4}
\item  for $f\in C_0^{\infty}(M)$,
  \begin{align*}
  &|d P_tf(x)|_{g^*}^2-\e^{\frac{K}{2}t}\E_x\left[\big|(\id+A_t)^*\ptr_t^{-1}df(X_t)\big|_{g^*}^2\right]\\
    &\quad \leq \E_x\left[\frac{K}{2}\int_0^t\e^{\frac{K}{2}(t+s)}|(\id+A_t-A_s)^*\ptr^{-1}_tdf(X_t)|_{g^*}^2\, ds\right];\\
    \intertext{and}
  &|2d f-d P_tf|_{g^*}^2(x)-\e^{\frac{K}{2}t}\E_x\left[\big|2df(x)-(\id+A_t)^*\ptr_t^{-1}df(X_t)\big|_{g^*}^2\right]\\
  &\quad \leq \E_x\left[\frac{K}{2}\int_0^t\e^{\frac{K}{2}(t+s)}|(\id+A_t-A_s)^*\ptr^{-1}_tdf(X_t)|_{g^*}^2\, ds\right].
  \end{align*}
\end{enumerate}
\end{theorem}
%We note the following immediate corollary.
%\begin{corollary}\label{low-bound-ineq} \todo[color=blue!20]{Moved from Ornstein-Uhlenbeck section. Why was it there? Should we delete the result inn general?}
% Assume that $-K \leq \Ric_H \leq K$ for
%some constant $K\geq 0$, then
%\begin{enumerate}[\rm (a)]
%\item for $f\in C_0^{\infty}(M)$ and $t>0$,
%  \begin{align*}
%  &P_tf^2-(P_tf)^2\\
%  &\quad \leq \frac{1}{2}\int_0^t(\e^{\frac{K}{2}(t+s)}+\e^{\frac{K}{2}(t-s)})\E_x\big|(\id+A_t-A_s)^*\ptr_t^{-1}df(X_t^x)\big|_{g^*}^2ds;
%  \end{align*}
%\item for $f\in C_0^{\infty}(M)$ and $t>0$,
%  \begin{align*}
%  &P_t(f^2\log f^2)-P_tf^2\log P_tf^2\\
%  &\quad \leq \int_0^t(\e^{\frac{K}{2}(t+s)}+\e^{\frac{K}{2}(t-s)})\E_x\big|(\id+A_t-A_s)^*\ptr_t^{-1}df(X_t^x)\big|_{g^*}^2ds.
% \end{align*}
%\end{enumerate}
%\end{corollary}
%\begin{proof}
%Let $T=t$. Taking  $F(\omega)=f(\omega_t)$ in Theorem \ref{main-th2} (iv)(v), we then obtain the results directly.
%\end{proof}

\begin{remark}\label{non-sym-bound}
Theorem \ref{main-th2} describes equivalent statements for a symmetric bound of $\Ric_H$.
For non-symmetric bounds, i.e.~if  $K_1 \leq \Ric_H \leq K_2$ for some constants $K_1 \leq K_2$, we can also give corresponding equivalent functional inequalities as in \cite{CT16} by modifying the gradient operator on the path space to:
\begin{align*}
\bar{D}_t F := \sum_{i=1}^n  1_{t \leq t_i}\e^{-\frac{K_1+K_2}{4}(t_i-t)}  \sharp (\id + A_{t_i} - A_t )^* \ptr_{t_i}^{-1}   d_iF.
\end{align*}
The functional inequalities will then be written in terms of $\bar{D}F$. We refer to \cite[Proposition 2.2]{CT16} for a more detailed discussion.
\end{remark}

\subsection{Proof of Theorem~\ref{main-th2}} We will show equivalence of the properties in Theorem~\ref{main-th2} by proving the relations
$$
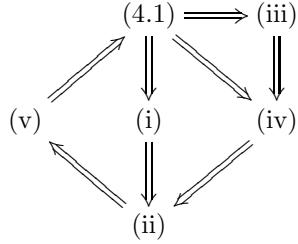
\xymatrix{
&  \text{\eqref{cur-cond}} \ar@{=>}[d] \ar@{=>}[rd] \ar@{=>}[r] & {\rm (iii)} \ar@{=>}[d] \\
{\rm (v)} \ar@{=>}[ru]  & {\rm (i)} \ar@{=>}[d] & {\rm  (iv)} \ar@{=>}[ld]  \\
 & {\rm (ii)} \ar@{=>}[lu]
}$$
We divide the proof into two parts.
\begin{proof}[Proof of Theorem \ref{main-th2}, Part I] In this part, we prove the implications ``\eqref{cur-cond} $\Rightarrow$ (i) $\Rightarrow$ (ii) $\Rightarrow$ (v)'', ``\eqref{cur-cond} $\Rightarrow$ (iii)'' and ``\eqref{cur-cond} $\Rightarrow$ (iv)''. \\
\paragraph{``\eqref{cur-cond}$\Rightarrow$ (i)"}
By Theorem \ref{IBPD} (a), we have
\begin{align*}
d_x\E_x[F]=\E_x[D_0F]-\frac{1}{2}\int_0^T\E[(\Ric_{\ptr_s}Q_s)^*D_sF]\,ds.
\end{align*}
Then using the condition \eqref{cur-cond}, we can prove (i) by controlling $Q_t$ and $\Ric_H$:
\begin{align*}
|d_x\E_x[F]| \leq|\E_x[D_0F]|+\frac{K}{2}\int_0^T\e^{\frac{1}{2}Ks}\E_x[|D_sF|]\,ds.
\end{align*}

\paragraph{``(i) $\Rightarrow$ (ii)"}
 It is easily observed that
\begin{align}\label{DQ}
\E_x\left[|\tilde D_tF |_{g}^2\Big|\mathscr{F}_t\right]&=\E_x\left[\Big(|D_tF|_{g}+\frac{K}{2}\int_t^T \e^{\frac{K}{2}(s-t)}|D_sF|_{g}ds\Big)^2\Big|\mathscr{F}_t\right]\notag\\
&\leq \e^{\frac{K}{2}(T-t)}\E_x\left[\Big(|D_tF|_{g}^2+ \frac{K}{2}\int_t^T \e^{\frac{K}{2}(s-t)}|D_sF|_{g}^2\ d s\Big)\Big|\mathscr{F}_t\right],
\end{align}
which implies (ii) with $t=0$.

\

\paragraph{``(ii)$\Rightarrow$ (v) "}  If $F(\omega)=f(\omega_t)$, then
$$D_0F(X_t)=\sharp(\id+A_t)^*\ptr_t^{-1}df(X_t), \quad D_sF(X_t)=\sharp (\id+A_t-A_s)^*\ptr_t^{-1}df(X_t).$$
 We obtain from (ii) that
\begin{align}\label{gra-eq1}
|dP_tf|_{g^*}^2 & \leq \e^{Kt}\E_x\left[|(\id+A_t)^*\ptr_t^{-1}df(X_t)|_{g^*}^2 \right] \\ \nonumber
& \qquad  +\frac{K}{2} \E\left[ \int_0^t\e^{\frac{K}{2}s}|(\id+A_t-A_s)^*\ptr_t^{-1}df(X_t)|_{g^*}^2ds\right].
\end{align}
Moreover, if $F(\omega)=2f(x)-f(\omega_t)$, then
\begin{align}\label{gra-eq2}
&|2d f-d P_tf|_{g^*}^2(x)-\e^{\frac{K}{2}t}\E_x\left[\big|2df(x)-(\id+A_t)^*\ptr_t^{-1}df(X_t)\big|_{g^*}^2\right]\notag\\
  &\quad \leq \E_x\left[\frac{K}{2}\int_0^t\e^{\frac{K}{2}(t+s)}|(\id+A_t-A_s)^*\ptr^{-1}_tdf(X_t)|_{g^*}^2\, ds\right].
\end{align}

\paragraph{``\eqref{cur-cond}$\Rightarrow$(iii)(iv)"}

We now prove
(iii) and (iv) by using the estimate \eqref{DQ} above and the It\^{o} formula,
\begin{align}\label{Log-S-I}
d( \mathbb{E}_x[F^2|\mathscr{F}_t]\log \mathbb{E}_x[F^2|\mathscr{F}_t])&=d M_t+ \frac{\big|\mathbb{E}_x[\tilde D_t F^2|\mathscr{F}_t] \big|_{g}^2}{2\mathbb{E}_x[F^2|\mathscr{F}_t]} dt\notag\\
&\leq d M_t+ 2 \E_x\left[ \big| \tilde D_tF\big|^2_{g} |\mathscr{F}_t\right]dt,
\end{align}
where $M_t$ is a local martingale such that
\begin{align*}
dM_t=(1+\log \mathbb{E}_x[F^2 |\mathscr{F}_t]) \langle \mathbb{E}(\tilde D_t F^2|\mathscr{F}_t), d B_t \rangle.
\end{align*}
 Integrating from $0$ to $t$ and taking expectation of both sides, we prove the inequality (iii).
Similarly, we can prove (iv) by taking into consideration of the process $\E_x[F|\mathscr{F}_t]^2$ and using the following It\^{o} formula:
\begin{align}\label{P-I}
d \E_x[F|\mathscr{F}_t]^2=d \tilde{M}_t+\E_x \left[\big| \tilde D_tF\big|^2_{g} |\mathscr{F}_t \right] dt,
\end{align}
where $\tilde{M}_t$ is a local martingale such that
\begin{align*}
d\tilde{M}_t=2 \mathbb{E}_x[F |\mathscr{F}_t] \langle  \mathbb{E}(\tilde D_t F|\mathscr{F}_t), d B_t \rangle.
\end{align*}
 Integrating from $0$ to $t$ and taking expectation of both sides, we prove the inequality (iv).
\end{proof}

To give the second part of the proof, we first need to include the following lemmas. For the Riemannian manifold case, the corresponding results can be found for instance in \cite{Wbook2,CT16}.
\begin{lemma}\label{B-W-Formula}
\emph{(Bochner-Weitzenb\"{o}ck formula)}
Let $L= \tr_H \nabla_{\times,\times}^2$ be the connection sub-Laplacian on tensors. For any $f\in C^{\infty}(M)$, we then have
\begin{align*}
L df(Z) - dLf(Z)& =  - 2 \tr_H \nabla_\times df(\mathbf{T}(\times, Z) ) + df(\Ric(Z) + \delta_H \mathbf{T}(Z)) .
\end{align*}
In particular,
\begin{align*}
\frac{1}{2} L|df|^2_{g^*} &= \langle dLf, df \rangle_{g^*}+ \langle (\Ric+ \delta_H \mathbf{T})^*(df), df \rangle_{g^*} \\
& \qquad+|\nabla df|_{g^* \otimes g^*}^2 -2 \tr_H \nabla_\times df(\mathbf{T}(\times, \sharp df)).
\end{align*}
\end{lemma}

\begin{proof}
The result follows from Lemma~\ref{lemma:Weitz}, Appendix, for the case of a metric preserving complement and from the property of the torsion for our choice of connection. \end{proof}

\begin{lemma}\label{prop1}
For $x\in M$, let $f\in C_0^{\infty}(M)$ be such that $\nabla df|_x = 0$ and $df(V_x) =0$.
Then the following limits hold:

  \begin{enumerate}[\rm (a)]
  \item
   $\displaystyle\frac{1}{2} \langle \sharp df, \Ric \sharp df\rangle_{g}(x)  =\lim_{t\downarrow 0}\frac{\big<d f, \E_x[(\id+A_t)^* \ptr_{t}^{-1} d f|_{X_t
  }]\big>_{g^*}(x)-\left<d f, d P_t f\right>_{g^*}(x)}{t};$
  \item 
  $\displaystyle\lim_{t\downarrow 0} \frac{1}{t^2} \mathbb{E}\left[ |A_t^*\ptr_t^{-1} df|_{g^*}^2 \right] =0;$
  \item 
    $\displaystyle\frac{1}{2} \langle \sharp df,  \Ric \sharp df\rangle_{g}(x)=\lim_{t\downarrow 0}\frac{\E_x\left[\big|(\id+A_t)^*\ptr_t^{-1}df|_{X_t}\big|_{g^*}^2\right]-|d P_{t}f|^2_{g^*}(x)}{t}.
    $
  \end{enumerate}
\end{lemma}
\begin{proof}
Let $ \nabla^Hf = \sharp df$ denote the horizontal gradient of a function $f$. Choose a taming Riemannian metric $\bar{g}$, i.e.~a Riemannian metric $\bar{g}$ such that $\bar{g}|_H = g$. We can always choose $\bar{g}$ so that $H$ and $V$ are orthogonal. We will use this taming metric to construct a relatively compact neighborhood of $t$ where we have reasonable estimates for the first exit time. If $\bar{\nabla} f$ is the gradient with respect to $\bar{g}$, we have that $\pr_H \bar{\nabla} f = \nabla^H f$. Let $d_{\bar{g}}$ be the Riemannian distance of $\bar{g}$. Choose sufficient small $r$ such that the ball $B_{\bar{g}}(x,r)$ of $\bar{g}$-radius $r$ centered at $x$ is outside the cut-locus of $x$.
 Let $\rho_t:=d_{\bar{g}}(x,X_t)$. Then
there exists a constant $c_1>0$ such that $Ld_{\bar{g}}^2(x,\cdot)\leq 2c_1$. By It\^{o}'s formula,
\begin{align*}
d \rho^2_t&=2\rho_t \langle \nabla^H \rho_t, \ptr_t d B_t \rangle_{g^*}
+ \frac{1}{2} L\rho^2_td t\\
&\leq 2\rho_t \langle \nabla^H\rho_t, \ptr_t d B_t \rangle_{g^*}
+c_1 dt,\quad  t<\sigma_r,
\end{align*}
where $\sigma_r= \inf\{t\geq 0: X_t \notin B_{\bar{g}}(x, r)\}$. For fixed $t>0$ and $\delta>0$, define
\begin{align*}
Z_s:=\exp{\left(\frac{\delta}{t}\rho_s^2-\frac{\delta }{t}c_1s-\frac{2\delta^2}{t^2}\int_0^s\rho_u^2 \ du \right)},\quad s\leq \sigma_r.
\end{align*}
Then as $|\nabla^H (d_{\bar{g}}(x,\cdot))|_{g} \leq | \bar{\nabla} (d_{\bar{g}}(x,\cdot))|_{\bar{g}} = 1$, we have
\begin{align*}
dZ_s&=\exp{\left(\frac{\delta}{t}\rho_s^2-\frac{\delta }{t}c_1s-\frac{2\delta^2}{t^2}\int_0^s\rho_u^2 \, du \right)}\\
 &\quad \times\left(\frac{2\delta}{t}\rho_s \langle\nabla^H\rho_s,\ptr_sd B_s \rangle_{g^*} -\frac{2\delta^2}{t^2}\rho_s^2\ d s+\frac{2\delta^2}{t^2}\rho_s^2|\nabla^H\rho_s|_{g^*}^2d s\right)\\
&\leq  \exp{\left(\frac{\delta}{t}\rho_s^2-\frac{\delta }{t}c_1s-\frac{2\delta^2}{t^2}\int_0^s\rho_u^2 \ du \right)} \left(\frac{2\delta}{t}\rho_s \langle\nabla^H \rho_s,\ptr_sd B_s \rangle_{g^*} \right),
\end{align*}
and hence $Z_s$ is a supermartingale. Therefore,
\begin{align*}
\mathbb{P}(\sigma_r\leq t)&=\mathbb{P}\left(\max_{s\in [0,t]}\rho_{s\wedge \sigma_r}\geq r\right)\\
&\leq \mathbb{P}\left(\max_{s\in [0,t]}Z_{s\wedge \sigma_r}\geq \e^{\delta r^2/t-\delta c_1-2\delta^2r^2/t}\right)\\
&\leq \exp \left(c_1\delta-\frac{1}{t}(\delta r^2-2\delta^2r^2)\right).
\end{align*}
If we take $\delta=1/4$, then
$$\mathbb{P}(\sigma_r\leq t)\leq \exp\left(\frac{1}{4}c_1-\frac{r^2}{8t}\right).$$
\begin{enumerate}[\rm (a)] 
\item 
By It\^o's formula and the estimate for $\sigma_r$,
  for small $t$ and $f\in C_0^{\infty}(M)$ such that
  $\nabla d f(x)=0$, we get
  \begin{align*}
    \E\big[\ptr_{t}^{-1}d f|_{X_t}\big]
    &= d f(x)+\E\left[\int_0^{t\wedge\sigma_r}\ptr_{s}^{-1} L (d f)|_{X_s}\, ds\right]+\text{\rm o}(t)\\
    &=d f|_x+ L (d f)|_x \,t+\text{\rm o}(t).
  \end{align*}
 From this equality,  the result
is easily derived using Taylor expansion:
  \begin{align*}
    &\big<d f, \E \ptr_{t}^{-1}df|_{X_t}\big>_{g^*}(x)-\left<d f, d P_tf\right>_{g^*}(x)\\
    &\quad=\big[\left<d f, Ld f\right>_{g^*}-\left<df, dLf\right>_{g^*}\big](x)t+\text{\rm o}(t)\\
    &\quad= \frac{1}{2} \langle \Ric^{*}df, df\rangle_g(x)\,t+\text{\rm o}(t),
  \end{align*}
where again we have used the Weitzenb\"ock identity and that $(\delta_H \mathbf{T})^* df|_x = 0$ since $df(V_x) = 0$.

Next, we observe from the Weitzenb\"ock identity that $L df| V = dLf|V$. Hence, the process $N_s$ in $T_x^* M$ given by $N_sv = \langle  \ptr_s^{-1} dP_{t-s} f|_{X_t}, v \rangle$ where $v \in V_x$, is a martingale. As a consequence, for any $v \in T_xM$,
\begin{align*}
& \mathbb{E}_x[\langle A_t^* \ptr_t^{-1} df|_{X_t}, v \rangle] = \mathbb{E}_x[\langle N_t, A_t v]\\
 &= \mathbb{E}_x\left[ \int_0^{t \wedge \sigma_r} \tr_H \langle  \nabla_{\ptr_s \times} dP_{t-s} f, \mathbf{T}(\ptr_s \times, \ptr_s v)  \rangle ds \right] + \text{o}(t),
\end{align*}
which is $\text{o}(t)$ since $\nabla df(x) = 0$.
\item Using that $\mathbb{E}[A_t^* \ptr_t^{-1} df|_{X_t} ] = \text{o}(t)$, that $df$ vanishes at $V_x$ and that $\nabla df =0$, we have
\begin{align*}
\qquad&\mathbb{E}_x\big[ | A_t^* \ptr_t^{-1} df|^2 \big]\\
& =  \mathbb{E}\left[ \int_0^{t\wedge \sigma_r}  \big|\ptr_s^{-1} \nabla_{\ptr_t ^{\mathstrut}\fatdot} df|_{X_s}( A_s \, \cdot \,) \big|^2_{g^* \otimes g^*} ds \right] + \mathbb{E}\left[ \int_0^{t\wedge \sigma_r}  \big|df \mathbf{T}(\ptr_s \,\cdot , \ptr_s \, \cdot) \big|^2 ds \right]  \\
& \qquad +2 \mathbb{E}\left[  \int_0^{t\wedge \sigma_r} \big\langle \nabla_{\ptr_s^{\mathstrut} \fatdot} df(A_s \,\cdot\,), df(\mathbf{T}(\ptr_s^{\mathstrut} \,\cdot ,  \ptr_s \,\cdot)) \big\rangle \right] + \text{o}(t^2) \\
& = \text{o}(t^2).
\end{align*}
As a consequence,
\begin{align*}
\lim_{t \downarrow 0} \frac{1}{t^2} \mathbb{E}\left[ |A_t^*\ptr_t^{-1} df|_{g^*}^2 \right] =0.
\end{align*}
\item
Since  $\nabla df(x) = 0$, we have $|\nabla df|_{g^* \otimes g^*}(x)=0$ and $$\tr_H \nabla_\times df(\mathbf{T}(\times, \sharp df))(x)=0.$$ By the Weitzenb\"ock formula in Proposition \ref{B-W-Formula},
\begin{align*}
\frac{1}{2}L|df|^2_{g^*}(x)-\left<d f, d Lf\right>_{g^*}(x)
&= \frac{1}{2} \langle{\rm Ric}^*(df),df\rangle_{g^*}(x).
\end{align*}
 Thus, by the Taylor expansions at the point $x$ (we drop $x$ below for simplicity):
\begin{align*}
\E_x\left[|\ptr_t^{-1}df(X_t)|^2_{g^*} \right]
& =|df|_{g^*}^2+ \frac{1}{2} L|df|_{g^*}^2\,t+{\rm o}(t),
\end{align*}
and
\begin{align*}
|dP_{t}f|_{g^*}^2&=|df|_{g^*}^2+\left<dLf, df\right>_{g^*} \, t+{\rm o}(t),\\
\end{align*}
we obtain 
\begin{align}\label{Ric-H-formula}
 \frac{1}{2} \langle \Ric^*df, df\rangle_{g^*}(x)
      &=\lim_{t\downarrow 0}\frac{\E_x\left[|\ptr_t^{-1}df|_{X_t}|^2_{g^*} \right]-|d P_{t}f|^2_{g^*}(x)}{t}.
\end{align}
On the other hand,
\begin{align*} 
\quad&\E_x\left[\big|(\id+A_t)^*\ptr_t^{-1}df|_{X_t}\big|_{g^*}^2\right] =\E_x\left[\big|(\id+A_t^*)\ptr_t^{-1}df|_{X_t}\big|_{g^*}^2\right]\\
&=\E_x\left[|\ptr_t^{-1}df|_{X_t}|^2_{g^*} \right]+2\E_x\langle \ptr_t^{-1}df||_{X_t}, A_t^*\ptr_t^{-1}df|_{X_t} \rangle+\E_x\left[|A_t^*\ptr_t^{-1}df|_{X_t}|^2\right]\\
&=\E_x\left[|\ptr_t^{-1}df(X_t)|^2_{g^*} \right]+2\E_x\langle \ptr_t^{-1}df(X_t)-dP_tf(x), A_t\ptr_t^{-1}df(X_t) \rangle\\
&\quad+\E_x\left[|A_t^*\ptr_t^{-1}df(X_t)|^2_{g^*}\right]\\
&=\E_x\left[|\ptr_t^{-1}df|_{X_t}|^2_{g^*} \right]+2\E_x\int_0^{t \wedge \sigma_r} \tr_H \langle  \nabla_{\ptr_s \times} dP_{t-s} f, \mathbf{T}(\ptr_s \times, \ptr_{t-s}^{-1}df|_{X_t}))  \rangle ds\\
&\quad+\E_x\left[|A_t^*\ptr_t^{-1}df|_{X_t}|^2_{g^*}\right] + \text{o}(t). 
\end{align*}
Since $\nabla d f=0$ and $df(V_x)=0$, then
\begin{align}\label{A-terms}
\lim_{t\downarrow 0}\left[\frac{2\E_x\int_0^{t\wedge \sigma_r} \tr_H \langle  \nabla_{\ptr_s \times} dP_{t-s} f, \mathbf{T}(\ptr_s \times, \ptr_{t-s}^{-1}df|_{X_t})  \rangle ds}{t} \right]=0.
\end{align}
Combining \eqref{A-terms}, (b) and \eqref{Ric-H-formula}, we finish the proof.
\qedhere
\end{enumerate}
\end{proof}

\begin{proof}[Proof of Theorem \ref{main-th2}, Part II]\ \\
We finish the remaining implications ``(iii)$\Rightarrow$(iv)$\Rightarrow$(ii)" and ``(v) $\Rightarrow$ \eqref{cur-cond}". \\

\paragraph{``(iv)$\Rightarrow$(ii)"}

By It\^{o}'s formula, we have
\begin{align}\label{PI-ineq}
&\E_x[\E_x[F|\mathscr{F}_t]^2]-\E_x[F]^2 =\int_0^t\E_x|D_s\E_x[F|\mathscr{F}_s]|^2\, ds\notag\\
&\leq \int_0^t\e^{\frac{K}{2}(T-s)}\left(\E_x|D_sF|_{g}^2+\frac{K}{2}\int_s^T\e^{\frac{K}{2}(r-s)}\E_x|D_rF|_{g}^2\, dr\right)\,ds.
\end{align}
For any function $F(\omega) = f(\omega_{t_1}, \dots, \omega_{t_n})$ with $t_1 >0$, the function $r\mapsto \E_x|D_rF|_{g}^2$
and $s\mapsto \E_x|D_s\E_x[F |\mathscr{F}_s]|^2$ are continuous at time $0$. Hence, we can divide both sides of \eqref{PI-ineq} by $t$ and take the limit as $t$ goes to $0$ to obtain (ii).

 If $t_1=0$, i.e. if $F(\omega)=f(x,\omega_{t_2},\ldots,\omega_{t_n})$, then
 $r\mapsto \E_x|D_rF|^2$ is not continuous at time $0$. We construct a family of functions
 $$F_{\varepsilon}(\omega)=f(\omega_\ve, \omega_{t_2},\ldots, \omega_{t_n})$$ for $0<\ve<t_2$.
 For this family of functions, $r\mapsto \E_x|D_rF_{\varepsilon}|^2$ is  continuous at time $0$ and
  \begin{align*}
   |d_x\mathbb{E}_x[F_{\varepsilon}]|_{g^*}^2\leq \e^{\frac{K}{2}T}\mathbb{E}_x\left[|D_0F_{\varepsilon}|_{g}^2+ \frac{K}{2}\int_0^T\e^{\frac{K}{2}s}|D_sF_{\varepsilon}|_{g}^2\ d s\right].
   \end{align*}
By considering the limit $\varepsilon \downarrow 0$, we prove (ii).

\

\paragraph{``(iii)$\Rightarrow$(iv)"} Applying the log-Sobolev inequality (iii) for $F^2=1+\ve G$, we have
\begin{align*}
      \quad&\E_x\Big[\E_x[(1+\ve G)|\mathscr{F}_{t}]\log \E_x[(1+\ve G)|\mathscr{F}_{t}]\Big]
             -\E_x[(1+\ve G)]\log \E_x[(1+\ve G)]\\
           &\leq 2\int_{0}^{t}\e^{\frac{K}{2}(T-r)}\left(\E_x| {D}_r\sqrt{1+\ve G}|^2_{g}+\frac{K}{2}\int_r^T\e^{\frac{K}{2}(s-r)}\E_x| {D}_s\sqrt{1+\ve G}|^2_{g}\ d s\right)\, d r.
    \end{align*}
Using the Taylor expansion at $\ve=0$, we have
\begin{align*}
  \quad&\E_x\Big[\ve \E[G|\mathscr{F}_t]+\frac{\ve^2}{2}\E[G|\mathscr{F}_t]^2\Big]-\Big[\ve\E_x(G)+\frac{\ve^2}{2}(\E_x G)^2\Big]+ \text{o}(\ve^2)\\
           &\leq 2\int_{0}^{t}\e^{\frac{K}{2}(T-r)}\left(\E_x| {D}_r\sqrt{1+\ve G}|^2_{g}+\frac{K}{2}\int_r^T\e^{\frac{K}{2}(s-r)}\E_x| {D}_s\sqrt{1+\ve G}|^2_{g}\ d s\right)\, d r\\
           &=\frac{1}{2}\int_{0}^{t}\e^{\frac{K}{2}(T-r)}\left(\ve^2\E_x|{D}_rG|^2_{g}+\frac{K}{2}\int_r^T\e^{\frac{K}{2}(s-r)}\ve^2\E_x| {D}_sG|^2_{g}\ d s\right)\, d r+\text{o}(\ve^2).
\end{align*}
Dividing both sides with $\ve^2$  and letting $\ve\rightarrow 0$, we then obtain (iv).

\
%
%\paragraph{``(iii)$\Rightarrow$\eqref{cur-cond}"}
%By Lemma \ref{B-W-Formula} (b), For any point $x \in M$ and any $\alpha \in H_x^*$ with $\alpha(V_x) =0$, we choose a function $f:M \to \mathbb{R}$ such that
%$$df(x) = \alpha, \qquad \nabla df(x) = 0, \qquad df(V) = 0.$$
%As $df(V)=0$, then $\langle (\delta_H\mathbf{T})^*(df), df\rangle=0$ and
%$$|\sharp(\id+A_t-A_s)^*\ptr^{-1}_tdf|_{g^*}^2=|\ptr^{-1}_tdf|_{g^*}^2$$
%for any $0\leq s\leq t\leq T$. Then the inequalities can be simplified for this function $f$:
% \begin{align*}
%  &|d P_tf(x)|_{g^*}^2-\e^{\frac{K}{2}t}\E_x\left[\big|\ptr_t^{-1}df(X_t)\big|_{g^*}^2\right]\leq \E_x\left[\frac{K}{2}\int_0^t\e^{\frac{K}{2}(t+s)}|\ptr^{-1}_tdf(X_t)|_{g^*}^2\, ds\right]
%  \end{align*}
%  and
%  \begin{align*}
%  &|2d f(x)-d P_tf(x)|_{g^*}^2-\e^{\frac{K}{2}t}\E_x\left[\big|2df(x)-\ptr_t^{-1}df(X_t)\big|_{g^*}^2\right]\\
%  &\leq \E_x\left[\frac{K}{2}\int_0^t\e^{\frac{K}{2}(t+s)}|\ptr^{-1}_tdf(X_t)|_{g^*}^2\, ds\right].
%  \end{align*}
%  Now using Lemma \ref{prop1}, we then obtain
%  $$-K|\alpha|^2(x)\leq \Ric(\alpha, \alpha)(x)\leq K|\alpha|^2(x).$$
%
\

\paragraph{``(v)$\Rightarrow$\eqref{cur-cond}"}
For any point $x \in M$ and any $\alpha \in H_x^*$ with $\alpha(V_x) =0$, we choose a function $f:M \to \mathbb{R}$ such that
$$df(x) = \alpha, \quad \nabla df(x) = 0.$$
We note that then $\langle (\delta_H\mathbf{T})^*(df), df\rangle(x)=0$. Observe also that the inequalities of (v) are equivalent to:
\begin{align*}
&\frac{|dP_tf|_{g^*}^2 -\E_x\left[|(\id+A_t)^*\ptr_t^{-1}df(X_t)|_{g^*}^2 \right]}{t}\\
 & \leq \frac{(\e^{Kt}-1)}{t}\E_x\left[|(\id+A_t)^*\ptr_t^{-1}df(X_t)|_{g^*}^2 \right] \\ \nonumber
& \qquad  +\frac{K}{2t} \E\left[ \int_0^t\e^{\frac{K}{2}s}|(\id+A_t-A_s)^*\ptr_t^{-1}df(X_t)|_{g^*}^2ds\right],
\end{align*}
and
\begin{align*}
&\frac{4(1-\e^{\frac{K}{2}t})}{t}|df|_{g^*}^2(x)-\frac{4(1-\e^{\frac{K}{2}t})}{t}\E_x\langle df(x), (\id+A_t)^*\ptr_t^{-1}df(X_t)\rangle\\
&-\frac{4\langle df(x), dP_tf(x) \rangle-4\E_x\langle df(x), (\id+A_t)^*\ptr_t^{-1}df(X_t)\rangle}{t}\\
&+\frac{|dP_tf|^2-\e^{\frac{K}{2}t}\E_x|(\id+A_t)^*\ptr_t^{-1}df(X_t)|_{g^*}^2}{t} \notag\\
  &\quad \leq \frac{K}{2t}\E_x\left[\int_0^t\e^{\frac{K}{2}(t+s)}|(\id+A_t-A_s)^*\ptr^{-1}_tdf(X_t)|_{g^*}^2\, ds\right].
\end{align*}
Letting $t$ tend to $0$ and using Lemma \ref{prop1}, we obtain
\begin{equation*}
-K|\sharp \alpha|^2_g(x)\leq \langle \sharp \alpha , \Ric \sharp \alpha \rangle_g (x)\leq K|\sharp \alpha|^2_g(x).\qedhere
\end{equation*}
\end{proof}

%Using Lemma \ref{prop1}, bounding the term $\E_x\left[\big\langle  A_t^*\ptr_t^{-1}df|_{X_t} , \ptr_t^{-1} df |_{X_t} \big\rangle_{g^*}\right]$ and using the Cauchy-Swartz inequality, we can divide by $t$ and let it tend to $0$ to finally obtain
%  $$-K|\alpha|^2(x)\leq \Ric(\alpha, \alpha)(x)\leq K|\alpha|^2(x).$$\qedhere

\subsection{The Ornstein-Uhlenbeck operator}
For cylindrical functions $F, G \in \FCc$ define a bilinear form
\begin{align*}
\mathscr{E}(F,G)&= \mathbb{E} \langle  DF,  DG \rangle_\mathbb{H}=\mathbb{E}\int_0^T \langle  D_tF,  D_t G \rangle_{g} \,d t.
\end{align*}
Then $(\mathscr{E}, \FCc)$ is a positive bilinear form on $L^2(W^T_x;\mathbb{P}^{x,T})$. It is standard that the integration by parts formula in Theorem~\ref{IBPD} implies closability of the form (see e.g. the argument in \cite[Lemma 4.3.1.]{Wbook2}). We shall use $(\mathscr{E},\Dom(\mathscr{E}))$ to denote the closure of $(\mathscr{E},\FCc)$. Let $(\mathscr{L}, \Dom(\mathscr{E}))$ be the analogue of the Ornstein-Uhlenbeck operator as the generator of the Dirichlet
form $\mathscr{E}$. Let $\gap(\mathscr{L})$ denote the spectral gap of the Ornstein-Uhlenbeck operator.  The following is then a consequence of Theorem~\ref{main-th2}.
\begin{corollary}\label{cor}
Assume that there exists  some non-negative  constant $K$ such that
\begin{align*}
-K \leq \Ric_H \leq K.
\end{align*}
  Then
\begin{enumerate}[\rm (a)]
 \item for any $F\in \Dom(\mathscr{E})$ with $\mathbb{E}_x[F]=0$,
  \begin{align*}
  \mathbb{E}_x(F^2)\leq \frac{1}{2}(\e^{KT}+1)\, \mathscr{E}_x(F,F);
  \end{align*}
  \item for any $F\in \Dom(\mathscr{E})$ with $\mathbb{E}_x[F^2]=1$,
  \begin{align*}
  \mathbb{E}_x(F^2\log F^2)\leq (\e^{KT}+1)\, \mathscr{E}_x(F,F);
  \end{align*}
 \item the spectral gap has the following estimate:
 \begin{align*}
 {\rm gap}(\mathscr{L})^{-1}\leq \frac{1}{2}(\e^{KT}+1).
 \end{align*}
\end{enumerate}
\end{corollary}
\begin{proof}
The inequalities in (a) and (b) are derived by using Theorem \ref{main-th2} (iii) and (iv) with $t=T$:
\begin{align*}
\E_x[F^2]-(\E_x[F])^2&\leq \int_0^T\e^{\frac{K}{2}(T-t)}\mathbb{E}_x\left[|D_tF|_{g}^2+ \frac{K}{2}\int_t^T\e^{\frac{K}{2}(s-t)}|D_sF|_{g}^2\ d s\right] dt\\
&\leq  \frac{1}{2}\ \mathbb{E}_x\left[\int_0^T(\e^{\frac{K}{2}(T+t)}+\e^{\frac{K}{2}(T-t)})|D_tF|_{g}^2\ d t\right]\\
&\leq \frac{1}{2}(\e^{KT}+1)\mathbb{E}_x\left[\int_0^T|D_tF|_{g}^2\ d t\right].
\end{align*}
The estimate in (c) is derived according to the definition of the spectral gap.
\end{proof}

%\begin{remark}\label{comments-on-non-compact-case}
%We continue to add some comments on non-compact case with less assumptions than in Section~\ref{sec:Noncompact}.
%When the non-compact sub-Riemannian manifold is complete
%and stochastic complete, we will see that $(\mathscr{E}, \FC_0)$
%is also closable in $L^2(W^T_x;\mathbb{P}^{x,T})$ by the same discussion as in
%\cite[Theorem 1.1]{ChWu14}. Note that the main idea of the approach is first  to localize the discussion
%on subset $B_R$, where
%\begin{align*}
%B_R:=\{x\in M: d_g(x,o)\leq R\},
%\end{align*}
%and then let $R$ tend to $\infty$. Moreover, if we introduce the following
%set
%\begin{align*}
%\FC_{loc}:=\{F\ell(\rho): F\in \FC_0,\ \ \ell\in C_0^{\infty}(\R)\},
%\end{align*}
%where $$\rho(\gamma)=\sup_{t\in [0,T]}d_g(\gamma(t),o).$$
%We can then show that $\ell(\rho)$ for $\ell\in C_0^{\infty}(\R)$ is in $\Dom(\mathscr{E})$ and the
%functional inequalities including Poincar\'{e} inequalities and log-Sobolev inequalities with respect to the function space $\FC_{loc}$ can be established
%by using the local information of the curvature.
%\end{remark}

\section{On the geometry of path space and complements} \label{sec:Geometry}
\subsection{Integrable complements}
Let $(M,H,g)$ be a sub-Riemannian manifold and let $V$ be a metric preserving complement that is also Frobenius integrable, i.e. $[V,V] \subseteq V$. Let $\Phi$ be the corresponding foliation of $V$. Then $M/\Phi$ locally has the structure of a Riemannian manifold. More precisely, any $x \in M$ has a neighborhood $U$ such that $\pi_U: U \to U/\Phi|_U$ is a submersion of differentiable manifolds. Since $V$ is metric preserving, there exists a Riemannian metric $\check{g}_U$ on $U/\Phi_U$ such that
$$\langle v, w \rangle_g = \langle (\pi_U)_* v, (\pi_U)_* w \rangle_{\check{g}_U}, \quad v,w \in H_x,\ x \in U.$$

Consider the special case when $\Phi$ is a regular foliation, i.e. when $\check{M} = M /\Phi$ is a differentiable manifold. Write the corresponding projection as $\pi: M \to \check{M}$. Let $\check{g}$ be the corresponding complete Riemannian metric on $\check{M}$. For the sake of simplicity, we assume that $(M, H, g)$ is complete, which implies that $(\check{M}, \check{g})$ is complete, as this is a distance decreasing map. Write $\check{\nabla}$ for the Levi-Civita connection of $\check{g}$. Let $x \in M$ be a given point with $\check{x} = \pi(x)$. Let $W_{x,H}^\infty$ and $W^\infty_{\check{x}}$ be the space of smooth curves with domain %no (
$[0, \infty)$ % no ]
starting at $x$ and $\check{x}$, respectively, where the curves starting at $x$ are required to be horizontal. Then since $H$ is an Ehresmann connection on $\pi$, curves starting at $\check{x}$ have unique horizontal lifts to $x$. Hence, the map $W_{x,H}^\infty \to W^\infty_{\check{x}}$, $\omega \mapsto \pi(\omega)$, is a bijection.

Next, let $hY$ denote the horizontal lift of a vector field $Y$ on $\check{M}$, that is $hY$ is the unique section of $H$ satisfying $\pi_* hY = Y$. We can then describe the connection $\nabla = \nabla^{g,V}$ as
\begin{equation} \label{NablaCheckNabla} \nabla_{hY} hY_2 = h \check{\nabla}_{Y} Y_2, \quad \nabla_Z hY =0, \quad \nabla_{hY}Z= [hY,Z],\end{equation}
for $Y, Y_2 \in \Gamma(T\check{M})$, $Z \in \Gamma(V)$. Hence, if $\Dev_{\check{x}}: T_{\check{x}} \check{M} \to \check{M}$ is the development map of $\check{\nabla}$, then we have the following commutative diagram:
$$\xymatrix{W_0^\infty(H_x) \ar[r]^-{\Dev_x} \ar[d]_{\pi_{*,x} } & W_{x,H}^\infty \ar[d]^\pi \\ W_0^\infty(T_{\check{x}} \check{M}) \ar[r]_-{\Dev_{\check{x}}} & W_{\check{x}}^\infty }$$
with every map in the diagram being a bijection.

Going from smooth curves to continuous curves, the concept of horizontal curves will no longer be well defined. However, if $\check{B}_t^{\check{x}}$ is the standard Brownian motion in $T_{\check{x}}\check{M}$ and $\check{X}_t^{\check{x}}$ denotes the Brownian motion in $\check{M}$, then we can still make sense of the following diagram
$$\xymatrix{B_t^x \ar@{|->}[r]^-{\Dev_x} \ar@{|->}[d]_{\pi_{*,x} } & X_t^x \ar@{|->}[d]^\pi \\  \check{B}_t^{\check{x}} \ar@{|->}[r]_-{\Dev_{\check{x}}} & \check{X}^{\check{x}}_t },$$

We finally note that by \eqref{NablaCheckNabla} we have $\Ric_H = \pi^* \check{\Ric}|_H$ where $\check{\Ric}$ denotes the Ricci operator on $\check{M}$. In summary, if we consider the path space $W_x(M)$ with the probability distribution given by the sub-Riemannian Brownian motion, then, viewed from the connection $\nabla = \nabla^{g,V}$, the path space has a geometry similar to that of the path space of $M/\exp(V)$. See \cite{GT16a} for more details.

\subsection{An instructive example}
Let $(M^{(1)},g^{(1)})$ and $(M^{(2)}, g^{(2)})$ be two oriented Riemannian manifolds, both of dimension $n$. Let $\SO(TM^{(1)})$ and $\SO(TM^{(2)})$ be the oriented orthonormal frame bundles. With respect to the diagonal action of $\SO(n)$ on $\SO(TM^{(1)}) \times \SO(T M^{(2)})$, we define
$$\mathbf{M} = (\SO(TM^{(1)}) \times \SO(TM^{(2)}))/\SO(n).$$
We can consider elements $q \in \mathbb{M}$ as linear isometries $q: T_{x^{(1)}}M^{(1)} \to T_{x^{(2)}} M^{(1)}$ where $(\varphi^{(1)}, \varphi^{(2)})\cdot \SO(n)$, $\varphi^{(1)} \in \SO(TM)_x$, $\varphi^{(2)} \in \SO(T M)_{x^{(2)}}$ can be identified with $q = \varphi^{(2)} \circ (\varphi^{(1)})^{-1}$.
Define $\pi^{(1)}: \mathbf{M} \to M^{(1)}$ and $\pi^{(2)}: \mathbf{M} \to M^{(2)}$ such that $q: T_{x^{(1)}} M^{(1)} \to T_{x^{(2)}} M^{(2)}$ is mapped to~$x^{(1)}$ and~$x^{(2)}$, respectively. We can then define a subbundle $H \subseteq T\mathbf{M}$ by
$$H = \left\{ \dot q_t \,\colon\, \begin{array}{l} \text{$\pi^{(1)}_* \dot q_t = \pi^{(2)}_* \dot q_t$,}\\ \text{$q_t Y_t$ is a parallel vector field for every parallel $Y_t$ along $\pi(q_t)$}\end{array}\right\}.$$
Then $H$ is an Ehresmann connection on both $\pi^{(1)}$ and $\pi^{(2)}$.  Furthermore, for any element $v \in H$, we have
$$|\pi^{(1)}_* v|^2_{g^{(1)}} = |\pi^{(2)}_* v |^2_{g^{(2)}} = : |v|^2_{\bf g}.$$
Consider the sub-Riemannian manifold $(\mathbf{M},H,\mathbf{g})$. This corresponds to the optimal control problem of rolling $M^{(1)}$ on $M^{(2)}$ without twisting or slipping along a minimizing curve. For more details and conditions for $H$ being bracket-generating see e.g. \cite{CK10,Gro12,Gro16}.

Consider the choices of complement $V^{(1)} = \ker \pi^{(1)}_*$ and $V^{(2)} = \ker \pi^{(2)}_*$. Both of $V^{(1)}$ and $V^{(2)}$ are metric preserving complements by definition. If $\nabla^{(1)} = \nabla^{\mathbf{g},V^{(1)}}$ and $\nabla^{(2)} = \nabla^{\mathbf{g},V^{(2)}}$ are the corresponding compatible connections with horizontal Ricci operator $\Ric_H^{(1)}$ and $\Ric_H^{(2)}$ respectively, then we have that $\Ric_H^{(1)} = (\pi^{(1)})^* \Ric_{g^{(1)}}$ and $\Ric^{(2)}_H = (\pi^{(2)})^* \Ric_{g^{(2)}}$, where $\Ric_{g^{(1)}}$ and $\Ric_{g^{(2)}}$ are the respective Ricci curvatures of $g^{(1)}$ and $g^{(2)}$. This illustrates that our formalism for path space of sub-Riemannian manifolds really depends on the choice of complementary subbundle.

\subsection{A non-integrable complement}\
We will include an example from~\cite{BeGr19}. Consider the Lie algebra $\mathfrak{so}(4)$. If $e_1, \dots, e_4$ is the standard basis of $\mathbb{R}^4$, we write $e_{ij} \in \mathfrak{so}(4)$ for the matrices satisfying $e_{ij}e_k = \delta_{ik} e_j - \delta_{jk} e_i$. Consider the inner product on $\mathfrak{so}(4)$ given by $\langle Y_1, Y_2\rangle = - \frac{1}{2} \tr Y_1Y_2$. Consider an orthogonal decomposition $\mathfrak{so}(4) = \mathfrak{h} \oplus \mathfrak{v}$ where $\mathfrak{v} =\spn \{ e_{12}, e_{23}\}$. On $\SO(4)$, define subbundles $T\SO(4) = H \oplus V$ where $H$ and $V$ are respective left translations of $\mathfrak{h}$ and $\mathfrak{v}$. We define a sub-Riemannian metric $g$ by left translation of the restriction of inner product of $\so(4)$ to $\mathfrak{h}$.

The subbundle $V$ is not integrable, but it is metric preserving from the bi-invariance of the inner product on $\so(4)$. Furthermore, if we define $\nabla = \nabla^{g,V}$, then
$$\frac{1}{2} \leq \Ric_H \leq 2.$$
See \cite[Example~3.1]{BeGr19} for detailed calculations.

\subsection{How to understand the curvature bounds} Let $(M,H,g)$ be a given sub-Riemannian manifold. As the above calculations show, our curvature bounds will in general depend on the choice of complement $V$ which determines the connection $\nabla = \nabla^{g,V}$. This dependence can be understood in the following way. Firstly, our connection sub-Laplacian $L = \tr_H \nabla^2_{\times,\times}$ depends on the choice of complement, and hence the same is true for the underlying diffusion $X_t^x$. See e.g. \cite{GoLa16,GT16a} for more details relating connections and sub-Laplacians. Furthermore, even for complements that define the same sub-Laplacian $L$,  the derivatives $\frac{d}{ds} \Dev(B_{\fatdot}^x - s k_{\fatdot})|_{s=0}$ on cylindrical functions will differ. In this sense, the curvature $\Ric_H$ can be seen as a curvature of the development map.

\appendix
\section{General geometric formulas} \label{sec:Any}
In most of our previous result, we restricted ourselves to sub-Riemannian manifolds $(M,H,g)$ equipped with a choice of metric preserving complement $V$. In this appendix, we include formulas without this assumption to point out additional complications that exist in general and for the benefit of future research.
\subsection{Weizenb\"ock formulas}
Let $(M,H,g)$ be a given sub-Riemannian manifold. Let $\nabla$ be an arbitrary connection with torsion $\mathbf{T}$ and curvature $\mathbf{R}$. Write
$$\delta_H \mathbf{T}(Z) = - \tr_H (\nabla_\times \mathbf{T})(\times ,Z), \quad \Ric(Z) = - \tr_H \mathbf{R}(\times, Z) \times.$$
Assume that $H$ is parallel with respect to $H$, and hence $\nabla g$ is well defined. For any vector field $Z \in \Gamma(TM)$, define $q_Z\colon H \to H$ by the formula
\begin{equation} \label{qmap} \langle q_Z v_1, v_2 \rangle_g = \frac{1}{2} (\nabla_Z g)(v_1, v_2), \quad v_1, v_2 \in H.\end{equation}
We note that $Z \mapsto q_Z$ is a tensorial map, so we can consider $q \in \Gamma(T^*M \otimes \End H)$ as a tensor.

\begin{lemma}[Weitzenb\"ock formula] \label{lemma:Weitz}
Let $\hat\nabla$ denote the adjoint connection of $\nabla$ as in \eqref{Adjoint}.
Write
$$L = \tr_H \nabla^2_{\times, \times} \quad \text{and} \quad \hat L = \tr_H \hat \nabla_{\times, \times}^2 $$
for the Laplacians on tensors. Then, for any function $f \in C^\infty(M)$, we have
\begin{align} \label{GenWeitzNabla}
L df(Z) - dLf(Z)& =  - 2 \tr_H \nabla_\times df(\mathbf{T}(\times, Z) - q_Z \times) \\ \nonumber
& \quad + df(\Ric(Z) + \delta_H \mathbf{T}(Z) - \tr_H \mathbf{T}(\times, \mathbf{T}(\times, Z))); \\
\hat L df(Z) - dLf(Z)& =  2 \tr_H \nabla_\times df( q_Z \times) + df(\Ric(Z) ) \label{GenWeitzHatNabla}.
\end{align}
\end{lemma}

\begin{proof}
For a given point $x$ and any elements $v \in H_x$ and $w \in T_xM$, choose arbitrary vector fields $Y \in \Gamma(H)$, $Z \in \Gamma(TM)$ such that $Y(x) =v$, $Z(x) = w$, $\nabla Y(x) = 0$ and $\nabla Z(x) = 0$. We remark that this is possible since we assumed that $H$ was parallel with respect to $\nabla$. Then at $x\in M$,
\begin{align*}
& (\nabla^2_{Y,Y} df)(Z)  = Y \nabla_Y df(Z) = Y \nabla_Z df(Y) + Y df(\mathbf{T}(Z,Y)) \\
& = (\nabla_{Y,Z}^2 df)(Y) - (\nabla_Y df)(\mathbf{T}(Y,Z)) - df( (\nabla_Y \mathbf{T})(Y,Z)) \\
& = (\nabla_{Z,Y}^2 df)(Y) + (\mathbf{R}(Y,Z) df)(A) - (\nabla_{\mathbf{T}(Y,Z)} df)(Y) \\
& \qquad - (\nabla_Y df)(\mathbf{T}(Y,Z)) - df( (\nabla_Y \mathbf{T})(Y,Z)) \\
& = (\nabla_{Z,Y}^2 df)(Y) - df(\mathbf{R}(Y,Z)Y)  \\
& \qquad - 2 (\nabla_Y df)(\mathbf{T}(Y,Z)) - df( (\nabla_Y \mathbf{T})(Y,Z) + \mathbf{T}(Y,\mathbf{T}(Y,Z))).
\end{align*}
Next, let us insert an orthonormal basis $Y_1, \dots, Y_k$ of $H$. We can choose this orthonormal basis such that $\nabla_Z Y_i(x) = q_Z Y_i(x)$ for some given point $x$. Evaluated at $x \in M$, we have
\begin{align*}
(\nabla_{Z,Y_i} df)(Y_i) & = Z ( \nabla_{Y_i} df)(Y_i) - (\nabla_{q_Z Y_i} df)(Y_i) - (\nabla_{Y_i} df)(q_Z Y_i) .
\end{align*}
Summing over this basis and using the symmetry of $q_Z$ gives us \eqref{GenWeitzNabla}. The result in \eqref{GenWeitzHatNabla} then follows from the identity
\begin{align*}
\hat \nabla_{Y,Y} df(Z) &= (\nabla_{Y,Y}^2 df)(Z) + 2 (\nabla_Y  df)(\mathbf{T}(Y,Z)) \\
& \qquad  + df((\nabla_Y \mathbf{T})(Y,Z)) +    df(\mathbf{T}(Y,\mathbf{T}(Y,Z))) .
\end{align*}
\end{proof}

\subsection{The smooth horizontal path space seen from an arbitrary complement} \label{sec:HorPathSmoothAny}
Let $(M, H, g)$ be a complete sub-Riemannian manifold and let $V$ be an arbitrary choice of complement. Let $\nabla = \nabla^{g,V}$ be the corresponding connection horizontally compatible with $(H,g)$ and with torsion $\mathbf{T}$ and curvature $\mathbf{R}$. Define the development map $\Dev$ relative to this connection. For any $Z \in \Gamma(TM)$, define $q_Z$ as in \eqref{qmap} and note that $q_{Z} = q_{\pr_V Z}$ since the connection i horizontally compatible. We note the following result.
\begin{lemma} \label{lemma:SmoothPathSpaceAny}
Let $t \mapsto \omega_t$ be a smooth horizontal curve with $u = \Dev^{-1}(\omega) \in W_0^\infty(H_x)$. Define $A_t = A_t^\omega \colon T_x M \to T_x M$ by $A_t = \int_0^t \mathbf{T}_{\ptr_s}(du_s, \, \cdot \,)$. Consider $\omega^s_t = \Dev(u+ sk)_t$ and define $Y_t = \frac{d}{ds} \omega^s_t |_{s=0}$. Then $Y_t = \ptr_t y_t = \hptr_t \hat y_t$ with
$$y_t = h_t + \int_0^t dA_s h_s, \quad \hat y_t = h_t - \int_0^t A_s dh_s.$$
where $h_t = \pr_H y_t$ is the solution of
\begin{align}\label{k-h-R}
k_t & = h_t  - \int_0^t\int_0^s \mathbf{R}_{\ptr_r} \left( du_r , h_r \right) du_s \notag\\
& \qquad  -\frac{1}{2} \int_0^t\int_0^s \left( (\nabla_{du_r} q)_{\ptr_r, \int_0^{r} dA_{r_2} h_{r_2}} du_s + (\nabla_{du_s} q)_{\ptr_r, \int_0^{r} dA_{r_2} h_{r_2}} du_r \right)\notag \\
& \qquad + \frac{1}{2}\int_0^t\int_0^s \sharp \left\langle (\nabla_{\!\fatdot} q)_{\ptr_r, \int_0^{r} dA_{r_2} h_{r_2}} du_r, du_s \right\rangle_g. \end{align}
\end{lemma}
\begin{proof}
Write $Y_t = \ptr_t y_t$, $\pr_H y_t = h_t$. Observe that from Lemma~\ref{lemma:SmoothPathSpace_0}, we must have
$$0 = \pr_V k_t = \pr_V y_t - \int_0^t dA_s h_s.$$
Then
$$y_t=\pr_Hy_t+\pr_Vy_t=h_t+\int_0^t dA_s h_s.$$
Furthermore, we have that
$$d(\ptr_t^{-1} \hptr_t) =  \ptr_t^{-1} \mathbf{T}(\ptr_t d u_t, \hptr_t \, \cdot \, ).$$
The solution of this equation is $\ptr_t^{-1} \hptr_t = \id + A_t$ and $\hptr_t^{-1} \ptr_t = \id- A_t$, since $A_t$ vanishes on $V$. As a consequence
$$\hat y_t = \hptr_t^{-1} \ptr_t y_t = h_t + \int_0^t dA_s h_s - A_t h_t = h_t - \int_0^t A_s dh_s. $$
Finally, we will prove \eqref{k-h-R} by first observing that
\begin{align}\label{k-h-R-1}
k_t&=h_t-\int_0^t\int_0^s\mathbf{R}_{\ptr_r}(du_r, y_r) du_s\notag\\
&=h_t-\int_0^t\int_0^s\mathbf{R}_{\ptr_r}(du_r, h_r) du_s-\int_0^t\int_0^s\mathbf{R}_{\ptr_r}(du_r, \pr_Vy_r) du_s.
\end{align}
Note that for arbitrary $z \in V_x$ and $v_1, v_2, v_3 \in H_x$, we have that
\begin{align*}
(\mathbf{R}(v_1, z)g)(v_2,v_3) &=2 \langle (\nabla_{v_1} q)_z v_2, v_3 \rangle_g = \langle \mathbf{R}(v_1, z) v_2, v_3 \rangle_g + \langle v_2, \mathbf{R}(v_1, z) v_3 \rangle_g,
\end{align*}
and hence from the first Bianchi identity
\begin{align*}
\langle \mathbf{R}(v_1, z) v_2, v_3 \rangle_g  &  = \langle \circlearrowright \mathbf{R}(v_1, z) v_2, v_2 \rangle_g  + \langle \mathbf{R}(v_2, z) v_1, v_3 \rangle_g \\
& = \langle \circlearrowright (\nabla_{v_1} \mathbf{T})( z, v_2) +  \circlearrowright \mathbf{T}(\mathbf{T}(v_1,z) v_2), v_3 \rangle_g  + \langle \mathbf{R}(v_2, z) v_1, v_3 \rangle_g \\
 &=  \langle \mathbf{R}(v_2, z) v_1, v_3 \rangle_g .
\end{align*}
Define $S_z(v_1, v_2, v_3) = \langle  \mathbf{R}(v_1, z) v_2 , v_3 \rangle_g$. We conclude that
\begin{align*}
\langle (\nabla_{v_1}q)_z v_2, v_3 \rangle_g & = S_z(v_1, v_2, v_3) + S_z(v_1, v_3, v_2), \\
0 & = S_z(v_1, v_2, v_3) - S_z (v_2, v_1, v_3).
\end{align*}
Considering
\begin{align*}
S_z(v_1, v_2, v_3) - S_z(v_2, v_1, v_3) & =0, \\
S_z(v_2, v_3, v_1) - S_z(v_3, v_2, v_1) & =0, \\
S_z(v_3, v_1, v_2) - S_z(v_1, v_3, v_2) & =0,
\end{align*}
and subtracting the second line from the sum of the first and the third, we obtain
\begin{align*}
S_z(v_1, v_2, v_3) - S_z(v_1, v_3, v_2) - \langle (\nabla_{v_2} q_z) v_1, v_3 \rangle + \langle (\nabla_{v_3} q)_z v_2, v_1 \rangle_g & =0.
\end{align*}
In conclusion
$$2S_z(v_1, v_2, v_3) = \langle (\nabla_{v_1} q)_z v_2, v_3 \rangle_g + \langle (\nabla_{v_2} q)_z v_1, v_3 \rangle_g - \langle (\nabla_{v_3} q)_z v_1, v_2 \rangle_g.$$
Combining this with the formula \eqref{k-h-R-1}, we prove \eqref{k-h-R}.
\end{proof}

\bibliographystyle{habbrv}
\bibliography{Bibliography}

\end{document}